\documentclass[11pt]{article}
\usepackage[tbtags]{amsmath}
\usepackage{amssymb}
\usepackage{amsthm}
\usepackage[misc]{ifsym}
\usepackage{cases}
\usepackage{mathrsfs}
\usepackage{color}
\numberwithin{equation}{section}   
\normalsize
\title{\bf A Global Maximum Principle for the Stochastic Optimal Control Problem with Delay \thanks{This work is financially supported by the National Key R\&D Program of China (2018YFB1305400), and the National Natural Science Foundations of China (11971266, 11571205, 11831010).}}
\author{\normalsize Weijun Meng\thanks{\it School of Mathematics, Shandong University, Jinan 250100, P.R. China, E-mail: 201611337@mail.sdu.edu.cn}, Jingtao Shi\thanks{\it Corresponding author, School of Mathematics, Shandong University, Jinan 250100, P.R. China, E-mail: shijingtao@sdu.edu.cn}}
\newtheorem{mypro}{Proposition}[section]
\newtheorem{mythm}{Theorem}[section]

\newtheorem{mylem}{Lemma}[section]
\newtheorem{Remark}{Remark}[section]
\begin{document}
\maketitle

\noindent{\bf Abstract:}\quad In this paper, an open problem is solved, for the stochastic optimal control problem with delay where the control domain is nonconvex and the diffusion term contains both control and its delayed term. Inspired by previous results by \O ksendal and Sulem [{\it A maximum principle for optimal control of stochastic systems with delay, with applications to finance. In J. M. Menaldi, E. Rofman, A. Sulem (Eds.), Optimal control and partial differential equations, ISO Press, Amsterdam, 64-79, 2000}] and Chen and Wu [{\it Maximum principle for the stochastic optimal control problem with delay and application, Automatica, 46, 1074-1080, 2010}], Peng's general stochastic maximum principle [{\it A general stochastic maximum principle for optimal control problems, SIAM J. Control Optim., 28, 966-979, 1990}] is generalized to the time delayed case, which is called the global maximum principle. A new backward random differential equation is introduced to deal with the cross terms, when applying the duality technique. Comparing with the classical result, the maximum condition contains an indicator function, in fact it is the characteristic of the stochastic optimal control problem with delay. The multi-dimensional case and a solvable linear-quadratic example are also discussed.

\vspace{2mm}

\noindent{\bf Keywords:}\quad Stochastic optimal control; stochastic differential delay equations; anticipated backward stochastic differential equation; maximum principle; linear-quadratic control

\vspace{2mm}

\noindent{\bf Mathematics Subject Classification:}\quad 93E20, 60H10, 34K50

\section{Introduction}

The study of stochastic optimal control problems has been an important topic in recent years. As one of the the main tools, Pontryagin's type maximum principle has drawn special attention to the researchers. Since Fleming \cite{Fleming69}, a lot of results have been obtained (see \cite{Kushner72}, \cite{Bismut78}, \cite{Ben82}, \cite{Peng90}, \cite{YZ99}). For stochastic optimal control problems with It\^o's type {\it stochastic differential equations} (SDEs for short), Kushner \cite{Kushner72} and Bismut \cite{Bismut78} proposed the definition of the adjoint processes. Bensoussan \cite{Ben82} gave the local maximum principle where either the control domain is convex or the diffusion term contains control variable with nonconvex control domain. Until 1990, Peng \cite{Peng90} proved the global maximum principle for the stochastic optimal control problem, applying the second-order spike variational technique of the optimal control and the second-order adjoint {\it backward stochastic differential equation} (BSDE for short). Systematically, Yong and Zhou \cite{YZ99} summarized these context.

In the above results, the states of the controlled systems only depend on the value of the current time. However, the development of some random phenomena in the real world depends not only on their current value, but also on their past history. Such past-dependence characteristic should be characterized by SDEs depending on the past, which are called {\it stochastic differential delay equations} (SDDEs for short). More detailed research about SDDEs can be referred to Mohammed \cite{Mohammed84,Mohammed98}, Mao \cite{Mao97}, and Kushner \cite{Kushner08}. Due to their wide applications in information science, engineering and finance (see \cite{OS00}, \cite{AHMP07}, \cite{KSW07}, \cite{CW10}, \cite{Fed11}, \cite{MS13}, \cite{SZ14}, \cite{ZXZ20}, etc.), the study of SDDEs has become a hot issue in modern research.

In the meanwhile, the delayed response brings many difficulties to study the stochastic control problems with delay, not only the corresponding problems become infinite dimensional ones, but also tools such as It\^o's formula to deal with the delay terms are absence so far. \O ksendal and Sulem \cite{OS00} studied a class of stochastic optimal control problem with delay. In their model, not only the current value but also the average value of the past duration will affect the growth of wealth at current time. Due to the particularity of the selected models, they could reduce the infinite dimensional problem into a finite dimensional one and obtain the sufficient maximum principle. Chen and Wu \cite{CW10} discussed the stochastic control system involving both delays in the state variable and the control variable with convex control domain, and they derived the local maximum principle by an {\it anticipated backward stochastic differential equation} (ABSDE for short) as the adjoint equation, where the ABSDE was first introduced by Peng and Yang \cite{PY09} in 2009. Recent progress for stochastic optimal control problems with delay, please refer to \cite{OSZ11}, \cite{Yu12}, \cite{CWY12}, \cite{DHQ13}, \cite{AHOP13}, \cite{ZZ15}, \cite{WW15}, \cite{ZX17}, \cite{WW18}, \cite{GM18}, \cite{LWW18}, \cite{Xu20} and the references therein.

Inspired by \cite{Peng90}, \cite{OS00} and \cite{CW10}, in this paper we derived the global maximum principle for a stochastic optimal control problem with delay. In our model, the drift and diffusion terms of the controlled system, and the running cost term of the cost functional could contain both the delay terms of the state and control variables, and the control domain is nonconvex. The contribution and innovation of this paper can be summarized as follows.

$\bullet$ The model in this paper is general, which covers the following special cases in the literatures: (1) Problems without delay, the control domain is convex and the diffusion term is control dependent (see Bensoussan \cite{Ben82}). (2) Problems without delay, the control domain is nonconvex and the diffusion term is control independent (see Bensoussan \cite{Ben82}). (3) Problems without delay, the control domain is nonconvex (see Peng \cite{Peng90}, Yong and Zhou \cite{YZ99}). (4) Problems with delay, the control domain is convex, the diffusion term is both control and control delay terms dependent, and the running cost term is state delay term independent (see Chen and Wu \cite{CW10}). (5) Problems with delay, the control domain is nonconvex, the drift, diffusion and running cost terms are control delay terms independent (see Guatteri and Masiero \cite{GM18}). It is worth to point that the authors of \cite{GM18} themselves declare that there is a mistake in their paper.

$\bullet$ The problem is open for over a decade, since there exist fatal, technical difficulties when applying the spike variational method of Peng \cite{Peng90} to the problem with delay. In fact, the key difficulty is how to deal with the cross term of state and state delay term appearing in the variational inequality. In this paper, we overcome the difficulty and solve the problem by introducing a new adjoint process $K(\cdot)$ which satisfies a {\it backward random differentia equation} (BRDE for short) (see (\ref{BRDE}) or (\ref{BRDE-multi})), which plays an important role in obtaining the indispensable estimation (see (\ref{variational inequality--}) or (\ref{variational inequality---multi})). Comparing with some existing results, the maximum condition contains an indicator function (see (\ref{main result}) or (\ref{main result-multi})), which is a natural characteristic of the stochastic optimal control problem with delay.

$\bullet$ Some estimations are given for the solution to the SDDEs, for the completeness of the content, which are supplement of existing results in the literature.

The rest of this paper is organized as follows. In Section 2, some preliminary results concerning the SDDEs and ABSDEs are presented. In Section 3, the stochastic optimal control problem with delay is formulated and the variational inequality is given. Section 4 mainly focuses on the adjoint equations and the global maximum principle. In section 5, the corresponding results are extended to the multi-dimensional case. In section 6, an {\it linear quadratic} (LQ for short) example is solved using the previous results. Finally, some concluding remarks are given in Section 7.

\section{Preliminaries}

In this section, we first present some preliminary results concerning the SDDE and ABSDE.

Suppose that $(\Omega,\mathcal{F},\{\mathcal{F}_t\},\mathbb{P})$ is a complete filtered probability space and $\{\mathcal{F}_t\}_{t\geq 0}$ is generated by the $d$-dimensional standard Brownian motion $\{B(t)\}_{t\geq0}$. Let $\mathbb{E}$ denote the mathematical expectation with respect to the probability $\mathbb{P}$ and $T>0$ be a given finite time duration.

For some Euclidean space $\mathbf{R}^n$, we first define some spaces which will be used later:
\begin{eqnarray*}\begin{aligned}
C([0,T];\mathbf{R}^n)&:=\Big\{\mathbf{R}^n\mbox{-valued continuous funciton }\phi(t);\sup\limits_{0\leq t\leq T}|\phi(t)|<\infty\Big\},\\
L^2([0,T];\mathbf{R}^n)&:=\Big\{\mathbf{R}^n\mbox{-valued funciton }\phi(t);\int_0^T|\phi(t)|^2<\infty\Big\},\\
L^2(\mathcal{F}_t;\mathbf{R}^n)&:=\Big\{\mathbf{R}^n\mbox{-valued } \mathcal{F}_t\mbox{-measurable random variable }\xi;\mathbb{E}|\xi(t)|^2<\infty\Big\},\ t\in[0,T],\\
\end{aligned}\end{eqnarray*}
\begin{eqnarray*}\begin{aligned}
L^2_\mathcal{F}([0,T];\mathbf{R}^n)&:=\Big\{\mathbf{R}^n\mbox{-valued }\mathcal{F}_t\mbox{-adapted process }\phi(t)\equiv
   \phi(t,\omega);\mathbb{E}\int_0^T|\phi(t)|^2dt<\infty\Big\},\\
S^2_\mathcal{F}([0,T];\mathbf{R}^n)&:=\Big\{\mathbf{R}^n\mbox{-valued }\mathcal{F}_t\mbox{-adapted process }\phi(t)\equiv
   \phi(t,\omega);\mathbb{E}\Big[\sup\limits_{0\leq t\leq T}|\phi(t)|^2\Big]<\infty\Big\}.
\end{aligned}\end{eqnarray*}

Consider the following SDDE:
\begin{eqnarray}\left\{\begin{aligned}\label{SDDE}
   dX(t)&=b(t,X(t),X(t-\delta))dt+\sigma(t,X(t),X(t-\delta))dB(t),\ t\geq 0,\\
    X(t)&=\varphi(t),\ t\in[-\delta,0],
\end{aligned}\right.\end{eqnarray}
where $\delta>0$ is a given finite time delay, $\varphi\in C([-\delta,0];\textbf{R})$ is the given initial path of the state $X(\cdot)$. $b:[0,T]\times\mathbf{R}^n\times\mathbf{R}^n\rightarrow\mathbf{R}^n,\sigma:[0,T]\times\mathbf{R}^n\times\mathbf{R}^n\rightarrow\mathbf{R}^{n\times d}$ are given functions satisfying:

$(\textbf{H1})\left|b\left(t,x,x^\prime\right)-b\left(t,y,y^\prime\right)\right|+\left|\sigma\left(t,x,x^\prime\right)-\sigma\left(t,y,y^\prime\right)\right|\leq D(|x-y|+|x^\prime-y^\prime|),\ \forall x,x^\prime,y,y^\prime\in\mathbf{R}^n,\ t\in[0,T]$, for some constant $D>0$;

$(\textbf{H2})\sup_{0\leq t\leq T}(|b(t,0,0)|+|\sigma(t,0,0)|)<+\infty$.

\vspace{1mm}

By Theorem 2.2 of Chen and Wu \cite{CW10}, we have the following result.

\begin{mypro}\label{pro2.1}
Suppose (\textbf{H1}) and (\textbf{H2}) hold, let $\varphi:\Omega\rightarrow C([-\delta,0];\mathbf{R}^n)$ is $\mathcal{F}_0$-measurable and $\mathbb{E}\Big[\sup\limits_{-\delta\leq t\leq0}|\varphi(t)|^2\Big]<\infty$. Then SDDE (\ref{SDDE}) admits a unique continuous $\mathcal{F}_t$-adapted solution $X(\cdot)\in\mathcal{S}^2_\mathcal{F}([0,T];\mathbf{R}^n)$.
\end{mypro}

In the following, an estimate for the solution to the SDDE (\ref{SDDE}) is given. Since it is fundamental in proving Lemma 4 and the main theorem of this paper, the detailed proof is given.

\begin{mylem}\label{lem2.1}
Suppose (\textbf{H1}) and (\textbf{H2}) hold, let $\varphi:\Omega\rightarrow C([-\delta,0];\mathbf{R}^n)$ is $\mathcal{F}_0$-measurable and $\mathbb{E}\Big[\sup\limits_{-\delta\leq t\leq 0}|\varphi(t)|^2\Big]<\infty$. Then for $p\geq2$, the solution to SDDE (\ref{SDDE}) satisfies the following estimate:
\begin{equation}\label{estimate}
\begin{aligned}
  \mathbb{E}\Big[\sup\limits_{0\leq t\leq T}|X(t)|^p\Big]
  &\leq C\mathbb{E}\bigg[\Big(\int_0^T|b(r,0,0)|dr\Big)^p+\Big(\int_0^T|\sigma(r,0,0)|^2dr\Big)^{\frac{p}{2}}+\sup\limits_{-\delta\leq r\leq 0}|\varphi(r)|^p\bigg],
\end{aligned}
\end{equation}
where $C\equiv C(\delta,T,D,p)$ is a constant depends on $\delta,T,D,p$.
\end{mylem}

\begin{proof}
By (\textbf{H1}) and Burkholder-Davis-Gundy's inequality, we have
\begin{equation*}\begin{aligned}
       &\mathbb{E}\Big[\sup\limits_{0\leq t \leq T}|X(t)|^p\Big]
        \leq C(p)\mathbb{E}\bigg(\int_0^T|b(r,X(r),X(r-\delta))|dr\bigg)^p\\
       &\ +C(p)\mathbb{E}\bigg(\int_0^T|\sigma(r,X(r),X(r-\delta))|^2dr\bigg)^{\frac{p}{2}}+C(p)\mathbb{E}|X(0)|^p\\
  \leq &\ C(p)\mathbb{E}|X(0)|^p+C(p)\mathbb{E}\bigg(\int_0^T\Big(|b(r,0,0)|+D|X(r)|+D|X(r-\delta)|\Big)dr\bigg)^p\\
       &+C(p)\mathbb{E}\bigg(\int_0^T\Big(|\sigma(r,0,0)|+D|X(r)|+D|X(r-\delta)|\Big)^2dr\bigg)^{\frac{p}{2}}\\
\end{aligned}\end{equation*}
\begin{equation*}\begin{aligned}
  \leq &\ C(p)\mathbb{E}|X(0)|^p+C(p)\mathbb{E}\Big(\int_0^T|b(r,0,0)|dr\Big)^p+C(p)\mathbb{E}\Big(\int_0^T|\sigma(r,0,0)|^2dr\Big)^{\frac{p}{2}}\\
       &+C(p,D,T)\mathbb{E}\int_0^T|X(r)|^pdr+C(p,D,T)\mathbb{E}\int_0^T|X(r-\delta)|^pdr.
\end{aligned}\end{equation*}
Noting
\begin{equation*}
  \mathbb{E}\int_0^T|X(r-\delta)|^pdr=\mathbb{E}\int_{-\delta}^{T-\delta}|X(r)|^pdr
  \leq\delta\mathbb{E}\bigg[\sup\limits_{-\delta\leq r\leq0}|\varphi(r)|^p\bigg]+\mathbb{E}\int_0^T|X(r)|^pdr,
\end{equation*}
hence we obtain
\begin{equation*}\begin{aligned}
       &\mathbb{E}\Big[\sup\limits_{0\leq t \leq T}|X(t)|^p\Big]
       \leq C(p,D,T,\delta)\mathbb{E}\Big[\sup\limits_{-\delta\leq r\leq 0}|\varphi(r)|^p\Big]+C(p)\mathbb{E}\bigg(\int_0^T|b(r,0,0)|dr\bigg)^p\\
       &\qquad +C(p)\mathbb{E}\bigg(\int_0^T|\sigma(r,0,0)|^2dr\bigg)^{\frac{p}{2}}+C(p,D,T)\mathbb{E}\int_0^T|X(r)|^pdr.
\end{aligned}\end{equation*}
Applying Gronwall's inequality, the proof of (\ref{estimate}) is completed.
\end{proof}

Let $\mathbf{R}^+$ be space of real numbers not less than zero. We consider the following ABSDE:
\begin{equation}\left\{\begin{aligned}\label{ABSDE}
  -dY(t)&=f\big(t,Y(t),Z(t),Y(t+\kappa(t)),Z(t+\zeta(t))\big)dt-Z(t)dB(t),\ t\in[0,T],\\
    Y(t)&=\mu(t),\ Z(t)=\nu(t),\ t\in[T,T+K].
\end{aligned}\right.\end{equation}
In the above, terminal conditions $\mu(\cdot)\in\mathcal{S}_\mathcal{F}^2\left([T,T+K];\mathbf{R}^m\right)$ and $\nu(\cdot)\in L_\mathcal{F}^2\left([T,T+K];\mathbf{R}^{m\times d}\right)$ are given, $\kappa(\cdot)$ and $\zeta(\cdot)$ are given $\mathbf{R}^+$-valued functions defined on $[0,T]$ satisfying:

\textbf{(H3)} (i) There exists a constant $K\geq0$ such that for all $s\in[0,T],\ s+\kappa(s)\leq T+K,\ s+\zeta(s)\leq T+K$;

(ii) There exists a constant $L\geq0$ such that for all $t\in[0,T]$ and for all nonnegative and integrable function $g(\cdot)$,
\begin{equation*}\begin{aligned}
  &\int_t^Tg(s+\kappa(s))ds\leq L\int_t^{T+K}g(s)ds,\\
  &\int_t^Tg(s+\zeta(s))ds\leq L\int_t^{T+K}g(s)ds.\\
\end{aligned}\end{equation*}

We impose the following conditions to the generator of ABSDE (\ref{ABSDE}):

\textbf{(H4)} $f(s,\omega,y,z,\kappa,\zeta):\Omega\times\mathbf{R}^m\times\mathbf{R}^{m\times d}\times L_{\mathcal{F}}^2([s,T+K];\mathbf{R}^m)\times L_\mathcal{F}^2([s,T+K];\mathbf{R}^{m\times d})\rightarrow L^2(\mathcal{F}_s;\mathbf{R}^m)$ for all $s\in[0,T]$ and $\mathbb{E}\big[\int_0^T|f(s,0,0,0,0)|^2ds\big]<+\infty$.

\textbf{(H5)} There exists a constant $C>0$ such that for all $s\in[0,T],\ y,y^\prime\in\mathbf{R}^m,\ z,z^\prime\in\mathbf{R}^{m\times d},\ \kappa,\kappa^\prime\in L_\mathcal{F}^2\left([s,T+K];\mathbf{R}^m\right),\ \eta,\eta^\prime\in L_\mathcal{F}^2\left([s,T+K];\mathbf{R}^{m\times d}\right),\ r,r^\prime\in[s,T+K]$, we have
\begin{equation*}
  \left|f\left(s,y,z,\kappa_r,\eta_{r^\prime}\right)-f\left(s,y^\prime,z^\prime,\kappa_r^\prime,\eta_{r^\prime}^\prime\right)\right|\leq C\left(\left|y-y^\prime\right|+\left|z-z^\prime\right|+\mathbb{E}^{\mathcal{F}_s}\left[\left|\kappa_r-\kappa_r^\prime\right|+\left|\eta_{r^\prime}-\eta_{r^\prime}^\prime\right|\right]\right),
\end{equation*}
where $\mathbb{E}^{\mathcal{F}_s}[\cdot]\equiv\mathbb{E}[\cdot|\mathcal{F}_s]$ denotes the conditional expectation, for $s\geq0$.

The following result can be found in Peng and Yang \cite{PY09}.
\begin{mypro}\label{pro2.3}
Let \textbf{(H3)}, \textbf{(H4)} and \textbf{(H5)} hold. Then for given $\mu(\cdot)\in\mathcal{S}_\mathcal{F}^2\left([T,T+K];\mathbf{R}^m\right)$ and $\nu(\cdot)\in L_\mathcal{F}^2\left([T,T+K];\mathbf{R}^{m\times d}\right)$, the ABSDE (\ref{ABSDE}) admits a unique $\mathcal{F}_t$-adapted solution pair $(Y(\cdot),Z(\cdot))\in\mathcal{S}_\mathcal{F}^2([0,T+K];\mathbf{R}^m)\times L_\mathcal{F}^2\left([0,T+K];\mathbf{R}^{m\times d}\right)$.
\end{mypro}

\begin{Remark}\label{rem2.1}
Note that the ABSDE (\ref{ABSDE}) is in a more general form with time-dependent delay terms $\kappa(\cdot),\zeta(\cdot)$, comparing with the SDDE (\ref{SDDE}) where only constant pointwise delay $\delta$ is considered. In fact, the ABSDEs introduced in our stochastic optimal control problem with delay in the next section are some special cases of (\ref{ABSDE}). We leave the result of \cite{PY09} untouched for the readers' convenience.
\end{Remark}

\section{Problem formulation and variational inequality}

In this section, we formulate the problem which will be studied and apply the spike variation technique to give the variational inequality.

Suppose $\mathbf{U}\subseteq\mathbf{R}^k$ is nonempty and nonconvex, let $\delta>0$ be a given constant time delay parameter, we consider the following stochastic control system with delay:
\begin{eqnarray}\left\{\begin{aligned}\label{controlled SDDE}
  dX(t)&=b(t,X(t),X(t-\delta),v(t),v(t-\delta))dt\\
       &\quad+\sigma(t,X(t),X(t-\delta),v(t),v(t-\delta))dB(t),\ t\geq 0,\\
   X(t)&=\varphi(t),\ v(t)=\eta(t),\ t\in[-\delta,0],
\end{aligned}\right.\end{eqnarray}
along with the cost functional
\begin{equation}\label{cost functional}
    J(v(\cdot))=\mathbb{E}\bigg[\int_0^Tl(t,X(t),X(t-\delta),v(t),v(t-\delta))dt+h(x(T))\bigg],
\end{equation}
where $b:[0,T]\times\mathbf{R}^n\times\mathbf{R}^n\times\mathbf{U}\times\mathbf{U}\rightarrow\mathbf{R}^n,\ \sigma:[0,T]\times\mathbf{R}^n\times\mathbf{R}^n\times\mathbf{U}\times\mathbf{U}\rightarrow\mathbf{R}^{n\times d},\ l:[0,T]\times\mathbf{R}^n\\\times\mathbf{R}^n\times\mathbf{U}\times\mathbf{U}\rightarrow\mathbf{R}$ and $h:\mathbf{R}^n\rightarrow\mathbf{R}$ are given functions.

We define the admissible controls set as follows:
$$\mathcal{U}_{ad}:=\Big\{v(\cdot)\big|v(\cdot)\mbox{ is a }\mathbf{U}\mbox{-valued, square-integrable},\ \mathcal{F}_t\mbox{-predictable process}\Big\}.$$

Our object is to find a control $u(\cdot)$ over $\mathcal{U}_{ad}$ such that (\ref{controlled SDDE}) is satisfied and (\ref{cost functional}) is minimized. Any $u(\cdot)\in \mathcal{U}_{ad}$ that achieves the above infimum is called an optimal control and the corresponding solution $x(\cdot)$ is called the optimal trajectory. $(u(\cdot),x(\cdot))$ is called an optimal pair.

Throughout the paper, we impose the following assumptions.

\textbf{(A1)}
(i) The functions $b=b(t,x,x_\delta,v,v_\delta),\ \sigma=\sigma(t,x,x_\delta,v,v_\delta)$ are twice continuously differentiable with respect to $(x,x_\delta)$ and their partial derivatives are uniformly bounded.

(ii) There exists a constant $C$ such that for $\phi=b,\sigma$,
\begin{equation*}
 |\phi(t,x,x_\delta,v,v_\delta)|\leq C(1+|x|+|x_\delta|),\ \forall x,x_\delta\in\mathbf{R}^n,\ v,v_\delta\in\mathbf{U},\ t\geq 0.
\end{equation*}

(iii) The function $\varphi:\Omega\rightarrow C([-\delta,0];\mathbf{R}^n)$ is $\mathcal{F}_0$-measurable and $\mathbb{E}\Big[\sup\limits_{-\delta\leq t\leq 0}|\varphi(t)|^2\Big]<\infty$. Meanwhile, the function $\eta:\Omega\rightarrow L^2([-\delta,0];\mathbf{R}^k)$ is $\mathcal{F}_0$-measurable and $\mathbb{E}\int_{-\delta}^0|\eta(t)|^2dt<\infty$.

(iv) $b,b_x,b_{x_\delta},b_{xx},b_{x_\delta x_\delta},b_{xx_\delta},\sigma,\sigma_x,\sigma_{x_\delta},\sigma_{xx},\sigma_{x_\delta x_\delta},\sigma_{xx_\delta}$ are continuous in $(x,x_\delta,v,v_\delta)$.

\vspace{1mm}

Under \textbf{(A1)}, for any admissible control $v(\cdot)$, (\ref{controlled SDDE}) admits a unique adapted solution $x(\cdot)\in\mathcal{S}_\mathcal{F}^2([0,T];\mathbf{R}^n)$ by Proposition \ref{pro2.1}.

The so-called global maximum principle, is some necessary conditions such that any optimal control should satisfies. In order to obtain it, we first introduce the variational equation, corresponding to (\ref{controlled SDDE}). Let $u(\cdot)$ be the optimal control and $x(\cdot)$ is the optimal trajectory. Since the control domain is nonconvex, as Peng \cite{Peng90}, we introduce the following spike variation of $u(\cdot)$. Let us define $u^\varepsilon(\cdot)$ as follows, which is a perturbed admissible control of the form:
\begin{eqnarray}\begin{aligned}\label{perturbed control}
u^\varepsilon(t)=
    \begin{cases}
    u(t), &t\notin[\tau,\tau+\varepsilon],\\
    v(t), &t\in[\tau,\tau+\varepsilon],
    \end{cases}\quad\mbox{ for\ all\ }t\in[0,T],
\end{aligned}\end{eqnarray}
and $x^\varepsilon(\cdot)$ is the corresponding trajectory, where $v(\cdot)$ is any admissible control.

For simplicity of presentation, in the following we only discuss the one-dimensional case, namely, $n=k=d=1$. However, the multi-dimensional case can be obtained without any difficulties, and we will present the corresponding results in Section 5.

We then introduce the first-order and second-order variational equations for $x(\cdot)$ as follows:
\begin{eqnarray}\left\{\begin{aligned}\label{variational equation-1}
    dx_1(t)=&\big[b_x(t)x_1(t)+b_{x_\delta}(t)x_1(t-\delta)+\Delta b(t)\big]dt\\
            &+\big[\sigma_x(t)x_1(t)+\sigma_{x_\delta}(t)x_1(t-\delta)+\Delta\sigma(t)\big]dB(t),\ t\geq0,\\
     x_1(t)=&\ 0,\ -\delta\leq t\leq 0,
\end{aligned}\right.\end{eqnarray}
\begin{eqnarray}\left\{\begin{aligned}\label{variational equation-2}
    dx_2(t)=&\big[b_x(t)x_2(t)+b_{x_{\delta}}(t)x_2(t-\delta)+\frac{1}{2}b_{xx}(t)|x_1(t)|^2+\frac{1}{2}b_{x_\delta x_\delta}(t)|x_1(t-\delta)|^2\\
            &\quad+b_{xx_\delta}(t)x_1(t)x_1(t-\delta)\big]dt\\
            &+\big[\sigma_x(t)x_2(t)+\sigma_{x_\delta}(t)x_2(t-\delta)+\frac{1}{2}\sigma_{xx}(t)|x_1(t)|^2+\frac{1}{2}\sigma_{x_\delta x_\delta}(t)|x_1(t-\delta)|^2\\
            &\quad+\sigma_{xx_\delta}(t)x_1(t)x_1(t-\delta))+\Delta\sigma_x(t)x_1(t)+\Delta\sigma_{x_\delta}(t)x_1(t-\delta)\big]dB(t),\ t\geq 0,\\
     x_2(t)=&\ 0,\ -\delta\leq t\leq 0,
\end{aligned}\right.\end{eqnarray}
where we denote $\Theta(t)\equiv(x(t),x(t-\delta),u(t),u(t-\delta))$ for simplicity and
\begin{eqnarray*}\left\{\begin{aligned}
          &b_x(t)=b_x(t,\Theta(t)),\ b_{x_\delta}(t)=b_{x_\delta}(t,\Theta(t)),\ b_{xx}(t)=b_{xx}(t,\Theta(t)),\ b_{xx_\delta}(t)=b_{xx_\delta}(t,\Theta(t)),\\
          &b_{x_\delta x_\delta}(t)=b_{x_\delta x_\delta}(t,\Theta(t)),\ \Delta b(t)=b(t,x(t),x(t-\delta),u^{\varepsilon}(t),u^{\varepsilon}(t-\delta))-b(t,\Theta(t)),\\
          &\Delta b_x(t)=b_x(t,x(t),x(t-\delta),u^{\varepsilon}(t),u^{\varepsilon}(t-\delta))-b_x(t,\Theta(t)),\\
          &\Delta b_{x_\delta}(t)=b_{x_\delta}(t,x(t),x(t-\delta),u^\epsilon(t),u^{\varepsilon}(t-\delta))-b_{x_\delta}(t,\Theta(t)),\\
\end{aligned}\right.\end{eqnarray*}
and $\sigma_x(t),\sigma_{x_\delta}(t),\sigma_{xx}(t),\sigma_{x_\delta x_\delta}(t),\sigma_{xx_\delta}(t),\Delta \sigma(t),\Delta \sigma_x(t),\Delta \sigma_{x_\delta}(t)$ can be defined similarly.

Under the assumption (\textbf{A1}), (i-iii), the variation equations (\ref{variational equation-1}) and (\ref{variational equation-2}) admit a unique solution $x_1(\cdot),x_2(\cdot)\in\mathcal{S}^2_\mathcal{F}([0,T];\mathbf{R}^n)$, respectively.

Note that the first variation equation (\ref{variational equation-1}) is different from (6) in Chen and Wu \cite{CW10}, where the control domain $\textbf{U}$ is convex. In fact, since in this paper $\textbf{U}$ is nonconvex, the above two variation equations (\ref{variational equation-1}) and (\ref{variational equation-2}) are similar to the classical ones (5) and (6) in Peng \cite{Peng90}, except for the cross terms like $b_{xx_\delta},\sigma_{xx_\delta}$ in (\ref{variational equation-2}). We point out that the cross terms appear naturally when we deal with the stochastic optimal control problem with delay, when applying the second-order Taylor's expansion to the state variable $X(\cdot)$ along the optimal trajectory $x(\cdot)$.

In the following, we introduce some estimates and their proofs are given since they are not completely standard.
\begin{mylem}\label{lem3.1}
Let assumption \textbf{(A1)} hold. Suppose $x(\cdot)$ is the optimal trajectory, $x^\varepsilon(\cdot)$ is the trajectory corresponding to $u^\varepsilon(\cdot)$, then for any $p\geq1$,
\begin{equation}\label{estimate-1}
      \mathbb{E}\Big[\sup_{0\leq t\leq T}|x^\varepsilon(t)-x(t)|^{2p}\Big]=O(\varepsilon^p),
\end{equation}
\begin{equation}\label{estimate-2}
      \mathbb{E}\Big[\sup_{0\leq t\leq T}|x_1(t)|^{2p}\Big]=O(\varepsilon^p),
\end{equation}
\begin{equation}\label{estimate-3}
      \mathbb{E}\Big[\sup_{0\leq t\leq T}|x_2(t)|^p\Big]=O(\varepsilon^p),
\end{equation}
\begin{equation}\label{estimate-4}
      \mathbb{E}\Big[\sup_{0\leq t\leq T}|x^\varepsilon(t)-x(t)-x_1(t)|^{2p}\Big]=o(\varepsilon^p),
\end{equation}
\begin{equation}\label{estimate-5}
      \mathbb{E}\Big[\sup_{0\leq t\leq T}|x^\varepsilon(t)-x(t)-x_1(t)-x_2(t)|^p\Big]=o(\varepsilon^p).
\end{equation}
\end{mylem}
\begin{proof}
For the simplicity of presentations, let $E_\varepsilon=[\tau,\tau+\varepsilon]\bigcup[\tau+\delta,\tau+\delta+\varepsilon]$. Noting when $t\in E_\varepsilon$, we have $\Delta b(t)\neq 0$, etc. If we choose $\varepsilon$ enough small, then $|E_\varepsilon|=2\varepsilon$. In the whole proof, $C>0$ is a generic constant, which change from line to line.

First, (\ref{estimate-1}) can be proved by a standard argument, so we omit the details. Next, recall Lemma \ref{lem2.1}, applying the assumption \textbf{(A1)} (ii), we obtain
\begin{equation*}\begin{aligned}
        &\mathbb{E}\Big[\sup\limits_{0\leq t \leq T}|x_1(t)|^{2p}\Big]
         \leq C\mathbb{E}\bigg(\int_0^T|\Delta b(t)|dt\bigg)^{2p}+C\mathbb{E}\bigg(\int_0^T|\Delta \sigma(t)|^2dt\bigg)^p\\
    \leq&\ C\mathbb{E}\bigg(\int_{E_\varepsilon}\big[1+|x(t)|+|x(t-\delta)|+|x^\varepsilon(t)|+|x^\varepsilon(t-\delta)|\big]dt\bigg)^{2p}\\
        &+C\mathbb{E}\bigg(\int_{E_\varepsilon}\big[1+|x(t)|+|x(t-\delta)|+|x^\varepsilon(t)|+|x^\varepsilon(t-\delta)|\big]^2dt\bigg)^p\\
    \leq&\ C(\varepsilon^p+\varepsilon^{2p})\mathbb{E}\Big[1+\sup\limits_{-\delta\leq t \leq T}|x(t)|^{2p}+\sup\limits_{-\delta\leq t \leq T}|x^\varepsilon(t)|^{2p}\Big]\\
    \leq&\ C(\varepsilon^p+\varepsilon^{2p})\mathbb{E}\Big[1+\sup\limits_{-\delta\leq t \leq 0}|\varphi(t)|^{2p}\Big]=O(\varepsilon^p).
\end{aligned}\end{equation*}
By (\ref{estimate}), the proof of (\ref{estimate-2}) is completed. Similarly, with assumption (\textbf{A1}) (i), we deduce
\begin{equation*}\begin{aligned}
        &\mathbb{E}\Big[\sup\limits_{0\leq t\leq T}|x_2(t)|^p\Big]\\
    \leq&\ C\mathbb{E}\bigg(\int_0^T\Big|\frac{1}{2}b_{xx}(t)|x_1(t)|^2+\frac{1}{2}b_{x_\delta x_\delta}(t)|x_1(t-\delta)|^2
         +b_{xx_\delta}(t)x_1(t)x_1(t-\delta)\Big|dt\bigg)^p\\
        &+C\mathbb{E}\bigg(\int_0^T\Big|\frac{1}{2}\sigma_{xx}(t)|x_1(t)|^2+\frac{1}{2}\sigma_{x_\delta x_\delta}(t)|x_1(t-\delta)|^2+\sigma_{xx_\delta}(t)x_1(t)x_1(t-\delta)\\
        &\qquad+\Delta\sigma_x(t)x_1(t)+\Delta\sigma_{x_\delta}(t)x_1(t-\delta)\Big|^2dt\bigg)^\frac{p}{2}\\
    \leq&\ C\mathbb{E}\Big[\sup\limits_{-\delta\leq t\leq T}|x_1(t)|^{2p}\Big]
         +C\mathbb{E}\bigg[\sup\limits_{0\leq t\leq T}|x_1(t)|^p\Big(\int_0^T|\Delta\sigma_x(t)|^2dt\Big)^{\frac{p}{2}}\bigg]\\
        &+C\mathbb{E}\bigg[\sup\limits_{-\delta\leq t\leq T}|x_1(t)|^p\Big(\int_0^T|\Delta\sigma_{x_\delta}(t)|^2dt\Big)^{\frac{p}{2}}\bigg]
         \leq C\varepsilon^p+C\varepsilon^{\frac{p}{2}}\varepsilon^{\frac{p}{2}}=C\varepsilon^p.
\end{aligned}\end{equation*}
Thus (\ref{estimate-3}) holds. In the following, we try to prove (\ref{estimate-4}) and (\ref{estimate-5}). Write
$$\xi(t):=x^\varepsilon(t)-x(t)-x_1(t),\quad \eta(t):=\xi(t)-x_2(t),$$
then $\xi(\cdot)$ and $\eta(\cdot)$ satisfy the following SDDEs, respectively:
\begin{equation}\left\{\begin{aligned}\label{SDDEs-1}
    d\xi(t)=&\Big\{b_x(t)\xi(t)+b_{x_\delta}(t)\xi(t-\delta)+\big[b_x^\theta(t)-b_x(t)\big](x^\varepsilon(t)-x(t))\\
            &\quad+\big[b_{x_\delta}^\theta(t)-b_{x_\delta}(t)\big](x^\varepsilon(t-\delta)-x(t-\delta))\Big\}dt\\
            &+\Big\{\sigma_x(t)\xi(t)+\sigma_{x_\delta}(t)\xi(t-\delta)+\big[\sigma_x^\theta(t)-\sigma_x(t)\big](x^\varepsilon(t)-x(t))\\
            &\quad+\big[\sigma_{x_\delta}^\theta(t)-\sigma_{x_\delta}(t)\big](x^\varepsilon(t-\delta)-x(t-\delta))\Big\}dB(t),\ t\geq0,\\
     \xi(t)=&\ 0,\ t\in[-\delta,0],
\end{aligned}\right.\end{equation}
\begin{equation}\left\{\begin{aligned}\label{SDDEs-2}
    d\eta(t)=&\Big\{b_x(t)\eta(t)+b_{x_\delta}(t)\eta(t-\delta)+\Delta b_x(t)(x^\varepsilon(t)-x(t))\\
             &\quad+\Delta b_{x_\delta}(t)(x^\varepsilon(t-\delta)-x(t-\delta))+\tilde{b}_{xx}(t)\big[|x^{\varepsilon}(t)-x(t)|^2-|x_1(t)|^2\big]\\
             &\quad+\big[\tilde{b}_{xx}(t)-\frac{1}{2}b_{xx}(t)\big]|x_1(t)|^2+\big[2\tilde{b}_{xx_\delta}(t)-b_{xx_\delta}(t)\big]x_1(t)x_1(t-\delta)\\
             &\quad+\big[\tilde{b}_{x_\delta x_\delta}(t)-\frac{1}{2}b_{x_\delta x_\delta}(t)\big]|x_1(t-\delta)|^2\\
             &\quad+\tilde{b}_{x_\delta x_\delta}(t)\big[|x^{\varepsilon}(t-\delta)-x(t-\delta)|^2-|x_1(t-\delta)|^2\big]\\
             &\quad+2\tilde{b}_{xx_\delta}(t)\big[(x^\varepsilon(t)-x(t))(x^\varepsilon(t-\delta)-x(t-\delta))-x_1(t)x_1(t-\delta)\big]\Big\}dt\\
             &+\Big\{\sigma_x(t)\eta(t)+\sigma_{x_\delta}(t)\eta(t-\delta)+\Delta\sigma_x(t)\xi(t)+\Delta\sigma_{x_\delta}(t)\xi(t-\delta)\\
             &\quad+\tilde{\sigma}_{xx}(t)\big[|x^\varepsilon(t)-x(t)|^2-|x_1(t)|^2\big]+\big[\tilde{\sigma}_{xx}(t)-\frac{1}{2}\sigma_{xx}(t)\big]|x_1(t)|^2\\
             &\quad+\tilde{\sigma}_{x_\delta x_\delta}(t)\big[|x^\varepsilon(t-\delta)-x(t-\delta)|^2-|x_1(t-\delta)|^2\big]\\
             &\quad+\big[\tilde{\sigma}_{x_\delta x_\delta}(t)-\frac{1}{2}\sigma_{x_\delta x_\delta}(t)\big]|x_1(t-\delta)|^2\\
             &\quad+2\tilde{\sigma}_{xx_\delta}(t)\big[(x^\varepsilon(t)-x(t))(x^\varepsilon(t-\delta)-x(t-\delta))-x_1(t)x_1(t-\delta)\big]\\
             &\quad+\big[2\tilde{\sigma}_{xx_\delta}(t)-\sigma_{xx_\delta}(t)\big]x_1(t)x_1(t-\delta)\Big\}dB(t),\ t \geq 0,\\
     \eta(t)=&\ 0,\ t\in[-\delta,0],
\end{aligned}\right.\end{equation}
where
\begin{equation*}\left\{\begin{aligned}
    b^\theta_x(t)&=\int_0^1 b_x(t,x(t)+\theta(x^\varepsilon(t)-x(t)),x(t-\delta)\\
                 &\qquad\qquad+\theta(x^\varepsilon(t-\delta)-x(t-\delta)),u^\varepsilon(t),u^\varepsilon(t-\delta))d\theta,\\
\tilde{b}_{xx}(t)&=\int_0^1\int_0^1 \theta b_{xx}(t,x(t)+\lambda\theta(x^\varepsilon(t)-x(t)),x(t-\delta)\\
                 &\qquad\qquad+\lambda\theta(x^\varepsilon(t-\delta)-x(t-\delta)),u^\varepsilon(t),u^\varepsilon(t-\delta))d\lambda d\theta,
\end{aligned}\right.\end{equation*}
and $b_{x_\delta}^\theta,\sigma_{x}^\theta,\sigma_{x_\delta}^\theta,\tilde{b}_{x_\delta x_\delta},\tilde{b}_{xx_\delta},\tilde{\sigma}_{xx},\tilde{\sigma}_{x_\delta x_\delta},\tilde{\sigma}_{xx_\delta}$ are similarly defined.

Applying Lemma \ref{lem2.1}, we get the following estimate with the assumption (\textbf{A1}) (iv):
\begin{equation*}\begin{aligned}
        &\mathbb{E}\Big[\sup\limits_{0\leq t\leq T}|\xi(t)|^{2p}\Big]\leq C\mathbb{E}\bigg(\int_0^T\big|(b_x^\theta(t)-b_x(t))(x^\varepsilon(t)-x(t))\\
        &\qquad+(b_{x_\delta}^\theta(t)-b_{x_\delta}(t))(x^\varepsilon(t-\delta)-x(t-\delta))\big|dt\bigg)^{2p}\\
        &\quad+C\mathbb{E}\bigg(\int_0^T\big[|\sigma_x^\theta(t)-\sigma_x(t)|^2|x^\varepsilon(t)-x(t)|^2\\
        &\qquad+|\sigma_{x_\delta}^\theta(t)-\sigma_{x_\delta}(t)|^2|x^\varepsilon(t-\delta)-x(t-\delta)|^2\big]dt\bigg)^p\\
\end{aligned}\end{equation*}
\begin{equation*}\begin{aligned}
    \leq&\ C\bigg\{\mathbb{E}\Big[\sup\limits_{0\leq t\leq T}|x^\varepsilon(t)-x(t)|^{4p}\Big]\bigg\}^{\frac{1}{2}}
         \bigg\{\mathbb{E}\Big[\Big(\int_0^T|b_x^\theta(t)-b_x(t)|dt\Big)^{4p}\\
        &\qquad+\Big(\int_0^T|b_{x_\delta}^\theta(t)-b_{x_\delta}(t)|dt\Big)^{4p}\Big]\bigg\}^{\frac{1}{2}}\\
        &+C\bigg\{\mathbb{E}\Big[\sup\limits_{0\leq t\leq T}|x^\varepsilon(t)-x(t)|^{4p}\Big]\bigg\}^{\frac{1}{2}}
         \bigg\{\mathbb{E}\Big[\Big(\int_0^T|\sigma_x^\theta(t)-\sigma_x(t)|^2dt\Big)^{2p}\\
        &\qquad+\Big(\int_0^T|\sigma_{x_\delta}^\theta(t)-\sigma_{x_\delta}(t)|^2dt\Big)^{2p}\Big]\bigg\}^{\frac{1}{2}}=o(\varepsilon^p).
\end{aligned}\end{equation*}
That is, (\ref{estimate-4}) holds. Finally, noting the boundedness of all the partial derivatives, we have
\begin{equation*}\begin{aligned}
    &\mathbb{E}\bigg(\int_0^T|\Delta b_x(t)(x^\varepsilon(t)-x(t))|dt\bigg)^p\\
\leq&\ \mathbb{E}\bigg[\sup\limits_{0\leq t\leq T}|x^\varepsilon(t)-x(t)|^p\Big(\int_0^T|\Delta b_x(t)|dt\Big)^p\bigg]=o(\varepsilon^p).
\end{aligned}\end{equation*}
For the same reason, we obtain
\begin{equation*}\begin{aligned}
    &\mathbb{E}\bigg(\int_0^T|\Delta b_{x_\delta}(t)(x^\varepsilon(t-\delta)-x(t-\delta))|dt\bigg)^p=o(\varepsilon^p),\\
    &\mathbb{E}\bigg(\int_0^T|\Delta\sigma_{x}(t)\xi(t)|^2dt\bigg)^{\frac{p}{2}}=o(\varepsilon^p),\quad
    \mathbb{E}\bigg(\int_0^T|\Delta\sigma_{x_\delta}(t)\xi(t-\delta)|^2dt\bigg)^{\frac{p}{2}}=o(\varepsilon^p).
\end{aligned}\end{equation*}
Using (\ref{estimate-1}) and (\ref{estimate-2}), we get
\begin{equation*}\begin{aligned}
         &\mathbb{E}\bigg(\int_0^T\tilde{b}_{xx}(t)\big[|x^\varepsilon(t)-x(t)|^2-|x_1(t)|^2\big]dt\bigg)^p\\
    \leq &\ C\mathbb{E}\bigg(\int_0^T\xi(t)\big[x^\varepsilon(t)-x(t)+x_1(t))\big]dt\bigg)^p\\
    \leq &\ C\mathbb{E}\bigg\{\sup\limits_{0\leq t\leq T}|\xi(t)|^p\bigg[\sup\limits_{0\leq t\leq T}|x^\varepsilon(t)-x(t)|^p+\sup\limits_{0\leq t\leq T}|x_1(t)|^p\bigg]\bigg\}
        =o(\varepsilon^p).
\end{aligned}\end{equation*}
In the same way, we have
\begin{equation*}\begin{aligned}
    &\mathbb{E}\bigg(\int_0^T\tilde{b}_{x_\delta x_\delta}(t)\big[|x^\varepsilon(t-\delta)-x(t-\delta)|^2-|x_1(t-\delta)|^2\big]dt\bigg)^p=o(\varepsilon^p),\\
    &\mathbb{E}\bigg(\int_0^T|\tilde{\sigma}_{xx}(t)|^2\big[|x^\varepsilon(t)-x(t)|^2-|x_1(t)|^2\big]^2dt\bigg)^{\frac{p}{2}}=o(\varepsilon^p),\\
    &\mathbb{E}\bigg(\int_0^T|\tilde{\sigma}_{x_\delta x_{\delta}}(t)|^2\big[|x^\varepsilon(t-\delta)-x(t-\delta)|^2-|x_1(t-\delta)|^2\big]^2dt\bigg)^{\frac{p}{2}}=o(\varepsilon^p).
\end{aligned}\end{equation*}
By (\ref{estimate-1}), (\ref{estimate-2}) and (\ref{estimate-4}), we have
\begin{equation*}\begin{aligned}
         &\mathbb{E}\bigg(\int_0^T\tilde{b}_{xx_{\delta}}(t)\big[(x^\varepsilon(t)-x(t))(x^\varepsilon(t-\delta)-x(t-\delta))-x_1(t)x_1(t-\delta)\big]dt\bigg)^p\\
    \leq &\ C\mathbb{E}\bigg(\int_0^T\big[\xi(t)(x^\varepsilon(t-\delta)-x(t-\delta))+x_1(t)\xi(t-\delta)\big]dt\bigg)^p\\
    \leq &\ C\mathbb{E}\bigg[\sup\limits_{0\leq t\leq T}|\xi(t)|^p\sup\limits_{-\delta\leq t\leq T}|x^\varepsilon(t)-x(t)|^p\bigg]
          +C\mathbb{E}\bigg[\sup\limits_{0\leq t\leq T}|x_1(t)|^p\sup\limits_{-\delta\leq t\leq T}|\xi(t)|^p\bigg]=o(\varepsilon^p),
\end{aligned}\end{equation*}
and similarly
\begin{equation*}
    \mathbb{E}\bigg(\int_0^T|\tilde{\sigma}_{xx_{\delta}}(t)|^2\big[(x^\varepsilon(t)-x(t))
    (x^\varepsilon(t-\delta)-x(t-\delta))-x_1(t)x_1(t-\delta)\big]^2dt\bigg)^{\frac{p}{2}}=o(\varepsilon^p).
\end{equation*}
By the continuity of $b_{xx}$, we deduce
\begin{equation*}\begin{aligned}
         &\mathbb{E}\bigg(\int_0^T\big[\tilde{b}_{xx}(t)-\frac{1}{2}b_{xx}(t)\big]|x_1(t)|^2dt\bigg)^p\\
    \leq &\ \mathbb{E}\bigg[\sup\limits_{0\leq t\leq T}|x_1(t)|^{2p}\Big(\int_0^T\big[\tilde{b}_{xx}(t)-\frac{1}{2}b_{xx}(t)\big]dt\Big)^p\bigg]=o(\varepsilon^{p}).
\end{aligned}\end{equation*}
Using the same method, we get
\begin{equation*}\begin{aligned}
    &\mathbb{E}\bigg(\int_0^T\big[2\tilde{b}_{xx_\delta}(t)-b_{xx_\delta}(t)\big]x_1(t)x_1(t-\delta)dt\bigg)^p=o(\varepsilon^p),\\
    &\mathbb{E}\bigg(\int_0^T\big[\tilde{b}_{x_\delta x_\delta}(t)-\frac{1}{2}b_{x_\delta x_\delta}(t)\big]|x_1(t-\delta)|^2dt\bigg)^p=o(\varepsilon^{p}),\\
    &\mathbb{E}\bigg(\int_0^T|\tilde{\sigma}_{xx}(t)-\frac{1}{2}\sigma_{xx}(t)|^2|x_1(t)|^4dt\bigg)^{\frac{p}{2}}=o(\varepsilon^p),\\
    &\mathbb{E}\bigg(\int_0^T|\tilde{\sigma}_{x_\delta x_\delta}(t)-\frac{1}{2}\sigma_{x_\delta x_\delta}(t)|^2|x_1(t-\delta)|^4dt\bigg)^{\frac{p}{2}}=o(\varepsilon^p),\\
    &\mathbb{E}\bigg(\int_0^T|2\tilde{\sigma}_{xx_\delta}(t)-\sigma_{xx_\delta}(t)|^2|x_1(t)|^2|x_1(t-\delta)^2dt\bigg)^{\frac{p}{2}}=o(\varepsilon^p).
\end{aligned}\end{equation*}
According to all above estimates, by Lemma \ref{lem2.1}, the proof of (\ref{estimate-5}) is completed.
\end{proof}

Based on Lemma \ref{lem3.1}, from
\begin{equation*}
J(v(\cdot))-J(u(\cdot))\geq0,\quad \mbox{for all }v(\cdot))\in\mathcal{U}_{ad},
\end{equation*}
we can obtain the variational inequality. The detail is omitted.
\begin{mylem}\label{lem3.2}
Let assumption \textbf{(A1)} hold. Suppose $(u(\cdot),x(\cdot))$ is an optimal pair, $x^\varepsilon(\cdot)$ is the trajectory corresponding to  $u^\varepsilon(\cdot)$ by (\ref{perturbed control}), then the following variational inequality holds:
\begin{equation}\label{variational inequality}\begin{aligned}
   &\mathbb{E}\bigg\{\int_0^T\Big[\Delta l(t)+l_x(t)(x_1(t)+x_2(t))+l_{x_\delta}(t)(x_1(t-\delta)+x_2(t-\delta))\\
   &\qquad+\frac{1}{2}l_{xx}(t)|x_1(t)|^2+\frac{1}{2}l_{x_\delta x_\delta}(t)|x_1(t-\delta)|^2+l_{xx_\delta}(t)x_1(t)x_1(t-\delta)\Big]dt\\
   &\qquad+h_x(x(T))(x_1(T)+x_2(T))+\frac{1}{2}h_{xx}(x(T))|x_1(T)|^2\bigg\}\geq o(\varepsilon).
\end{aligned}\end{equation}
where $\Delta l,l_x,l_{x_\delta},l_{xx},l_{x_\delta x_\delta},l_{xx_\delta},h_x,h_{xx}$ are defined similarly as before.
\end{mylem}

\section{Global maximum principle}

Based on the results in the previous section, in this section we prove the global maximum principle, for our stochastic optimal control problem with delay. Suppose $u(\cdot)$ is the optimal control and $u^\varepsilon(\cdot)$ is the perturbed admissible control.

We introduce the following three adjoint equations:
\begin{eqnarray}\left\{\begin{aligned}\label{adjoint equations-1}
-dp(t)=&\Big\{b_x(t)p(t)+\sigma_x(t)q(t)+l_x(t)+\mathbb{E}^{\mathcal{F}_t}\big[b_{x_\delta}(t+\delta)p(t+\delta)\\
       &\quad+\sigma_{x_\delta}(t+\delta)q(t+\delta)+l_{x_\delta}(t+\delta)\big]\Big\}dt-q(t)dB(t),\ t\in[0,T],\\
  p(T)=&\ h_x(x(T)),\ p(t)=0,\ t\in(T,T+\delta],\ q(t)=0,\ t\in[T,T+\delta].
\end{aligned}\right.\end{eqnarray}
\begin{eqnarray}\left\{\begin{aligned}\label{adjoint equations-2}
-dP(t)=&\Big\{2b_x(t)P(t)+|\sigma_x(t)|^2P(t)+2\sigma_x(t)Q(t)+b_{xx}(t)p(t)+\sigma_{xx}(t)q(t)+l_{xx}(t)\\
       &\quad+\mathbb{E}^{\mathcal{F}_t}\big[|\sigma_{x_\delta}(t+\delta)|^2P(t+\delta)+b_{x_\delta x_\delta}(t+\delta)p(t+\delta)
        +\sigma_{x_\delta x_\delta}(t+\delta)q(t+\delta)\\
       &\quad+l_{x_\delta x_\delta}(t+\delta)\big]\Big\}dt-Q(t)dB(t),\ t\in[0,T],\\
  P(T)=&\ h_{xx}(x(T)),\ P(t)=0,\ t\in(T,T+\delta],\ Q(t)=0,\ t\in[T,T+\delta],
\end{aligned}\right.\end{eqnarray}
\begin{equation}\left\{\begin{aligned}\label{BRDE}
-dK(t)&=\Big\{b_{x_\delta}(t)P(t)+\sigma_x(t)\sigma_{x_\delta}(t)P(t)+\sigma_{x_\delta}(t)Q(t)+b_{xx_\delta}(t)p(t)\\
      &\qquad+\sigma_{xx_\delta}(t)q(t)+l_{xx_\delta}(t)\Big\}dt,\ t\in[0,T],\\
  K(t)&=0,\ t\in[T,T+\delta].
\end{aligned}\right.\end{equation}

In the above, comparing (\ref{adjoint equations-1}) and (\ref{adjoint equations-2}) with (19) and (20) in Peng \cite{Peng90}, we see that time-advanced or time-anticipated terms such as $\mathbb{E}^{\mathcal{F}_t}\big[b_{x_\delta}(t+\delta)p(t+\delta)]$, etc., appear in the generators of them. Thus, (\ref{adjoint equations-1}) and (\ref{adjoint equations-2}) are both ABSDEs. Since all the partial derivatives are uniformly bounded, by Proposition \ref{pro2.3}, (\ref{adjoint equations-1}) and (\ref{adjoint equations-2}) admit unique solution pairs $(p(\cdot),q(\cdot))\in\mathcal{S}_\mathcal{F}^2([0,T+\delta];\mathbf{R})\times L_\mathcal{F}^2([0,T+\delta];\mathbf{R})$ and $(P(\cdot),Q(\cdot))\in\mathcal{S}_\mathcal{F}^2([0,T+\delta];\mathbf{R})\times L_\mathcal{F}^2([0,T+\delta];\mathbf{R})$, respectively.

\begin{Remark}\label{rem4.1}
We point out that, (\ref{adjoint equations-1}) and (\ref{adjoint equations-2}) are introduced to deal with terms like $x_1(\cdot)+x_2(\cdot)$ and $|x_1(\cdot)|^2$ in the variational equations (\ref{variational equation-1}) and (\ref{variational equation-2}), respectively. This is quite similar as Peng \cite{Peng90} for the global maximum principle of stochastic optimal control problem without delay. However, for the stochastic optimal control problem with delay, when deriving the global maximum principle, we have to deal with the cross terms like $x_1(\cdot)x_1(\cdot-\delta)$ in the variational equation (\ref{variational equation-2}). Thus, the additional adjoint process $K(\cdot)$ satisfying the BRDE (\ref{BRDE}), is introduced, which is motivated by (2.8) of \O ksendal and Sulem \cite{OS00}. It plays an important role in the derivation of our main theorem.
\end{Remark}

The following task is to give the global maximum principle. Since $x_1(t)=0,\ t\in[-\delta,0]$ and $p(t)=q(t)=0,\ t\in(T,T+\delta]$, we have
\begin{equation*}\begin{aligned}
    &\mathbb{E}\int_0^Tx_1(t-\delta)\big[b_{x_\delta}(t)p(t)+\sigma_{x_\delta}(t)q(t)\big]dt\\
   =&\ \mathbb{E}\int_{-\delta}^{T-\delta}x_1(t)\big[b_{x_\delta}(t+\delta)p(t+\delta)+\sigma_{x_\delta}(t+\delta)q(t+\delta)\big]dt\\
   =&\ \mathbb{E}\int_0^Tx_1(t)\big[b_{x_\delta}(t+\delta)p(t+\delta)+\sigma_{x_\delta}(t+\delta)q(t+\delta)\big]dt\\
    &-\mathbb{E}\int_{T-\delta}^{T}x_1(t)\big[b_{x_\delta}(t+\delta)p(t+\delta)+\sigma_{x_\delta}(t+\delta)q(t+\delta)\big]dt\\
   =&\ \mathbb{E}\int_0^Tx_1(t)\big[b_{x_\delta}(t+\delta)p(t+\delta)+\sigma_{x_\delta}(t+\delta)q(t+\delta)\big]dt\\
    &-\mathbb{E}\int_T^{T+\delta}x_1(t-\delta)\big[b_{x_\delta}(t)p(t)+\sigma_{x_\delta}(t)q(t)\big]dt\\
   =&\ \mathbb{E}\int_0^Tx_1(t)E^{\mathcal{F}_t}[b_{x_\delta}(t+\delta)p(t+\delta)+\sigma_{x_\delta}(t+\delta)q(t+\delta)]dt.
\end{aligned}\end{equation*}
Similarly, suppose $l_{x_\delta}(t)=0$ for $t\in(T,T+\delta]$, then we have
\begin{equation*}
   \mathbb{E}\int_0^Tl_{x_\delta}(t)\big[x_1(t-\delta)+x_2(t-\delta)\big]dt=\mathbb{E}\int_0^T\mathbb{E}^{\mathcal{F}_t}\big[l_{x_\delta}(t+\delta)\big][x_1(t)+x_2(t)]dt.
\end{equation*}

Applying It\^o's formula to $p(\cdot)(x_1(\cdot)+x_2(\cdot))+\frac{1}{2}P(\cdot)|x_1(\cdot)|^2$ and substituting it into (\ref{variational inequality}), we obtain
\begin{equation}\begin{aligned}\label{variational inequality---}
   &\mathbb{E}\bigg\{\int_0^T\Big[\Delta l(t)+p(t)\Delta b(t)+q(t)\Delta\sigma(t)+\frac{1}{2}P(t)|\Delta\sigma(t)|^2+x_1(t-\delta)\big[q(t)\Delta\sigma_{x_\delta}(t)\\
   &\qquad+P(t)\sigma_{x_\delta}(t)\Delta\sigma(t)\big]+x_1(t)\big[q(t)\Delta\sigma_x(t)+P(t)\Delta b(t)+P(t)\sigma_x(t)\Delta\sigma(t)\\
   &\qquad+Q(t)\Delta\sigma(t)\big]+x_1(t)x_1(t-\delta)\big[l_{xx_\delta}(t)+b_{x_\delta}(t)P(t)+\sigma_x(t)\sigma_{x_\delta}(t)P(t)\\
   &\qquad+\sigma_{x_\delta}(t)Q(t)+b_{xx_\delta}(t)p(t)+\sigma_{xx_\delta}(t)q(t)]\Big]dt\bigg\}\geq o(\varepsilon).
\end{aligned}\end{equation}
Applying again It\^o's formula to $K(\cdot)x_1(\cdot)$, we get for $r\in[\delta,T]$,
\begin{equation}\begin{aligned}\label{K(t)-1}
  &K(r)x_1(r)-K(\delta)x_1(\delta)\\
 =&\int_\delta^r\Big\{-x_1(t)\big[b_{x_\delta}(t)P(t)+\sigma_x(t)\sigma_{x_\delta}(t)P(t)+\sigma_{x_\delta}(t)Q(t)+b_{xx_\delta}(t)p(t)\\
  &\quad+\sigma_{xx_\delta}(t)q(t)+l_{xx_\delta}(t)\big]+K(t)\big[b_x(t)x_1(t)+b_{x_{\delta}}(t)x_1(t-\delta)+\Delta b(t)\big]\Big\}dt\\
  &\ +\int_\delta^rK(t)\big[\sigma_x(t)x_1(t)+\sigma_{x_{\delta}}(t)x_1(t-\delta)+\Delta\sigma(t)\big]dB(t)\\
 =&\int_0^{r-\delta}\Big\{-x_1(t+\delta)\big[b_{x_\delta}(t+\delta)P(t+\delta)+\sigma_x(t+\delta)\sigma_{x_\delta}(t+\delta)P(t+\delta)\\
  &\quad+\sigma_{x_\delta}(t+\delta)Q(t+\delta)+b_{xx_\delta}(t+\delta)p(t+\delta)+\sigma_{xx_\delta}(t+\delta)q(t+\delta)+l_{xx_\delta}(t+\delta)\big]\\
  &\quad+K(t+\delta)\big[b_x(t+\delta)x_1(t+\delta)+b_{x_{\delta}}(t+\delta)x_1(t)+\Delta b(t+\delta)\big]\Big\}dt\\
  &\ +\int_0^{r-\delta}K(t+\delta)\big[\sigma_x(t+\delta)x_1(t+\delta)+\sigma_{x_{\delta}}(t+\delta)x_1(t)+\Delta\sigma(t+\delta)\big]dB(t+\delta).
\end{aligned}\end{equation}
If $K(t)=0$ for all $t\in[0,T]$, then (\ref{K(t)-1}) can be rewritten as
\begin{equation}\begin{aligned}\label{K(t)-2}
  &K(r-\delta)x_1(r-\delta)-K(0)x_1(0)\\
  =&-\int_0^{r-\delta}x_1(t+\delta)\big[b_{x_\delta}(t+\delta)P(t+\delta)+\sigma_x(t+\delta)\sigma_{x_\delta}(t+\delta)P(t+\delta)+\sigma_{x_\delta}(t+\delta)Q(t+\delta)\\
    &\qquad+b_{xx_\delta}(t+\delta)p(t+\delta)+\sigma_{xx_\delta}(t+\delta)q(t+\delta)+l_{xx_\delta}(t+\delta)\big]dt,\ \delta\leq r\leq T.
\end{aligned}\end{equation}
Noting (\ref{variational equation-1}), we have
\begin{equation}\begin{aligned}\label{eq}
  x_1(r-\delta)&=\int_0^{r-\delta}\big[b_x(t)x_1(t)+b_{x_{\delta}}(t)x_1(t-\delta)+\Delta b(t)\big]dt\\
               &\quad+\int_0^{r-\delta}\big[\sigma_x(t)x_1(t)+\sigma_{x_\delta}(t)x_1(t-\delta)+\Delta\sigma(t)\big]dB(t).
\end{aligned}\end{equation}
By (\ref{K(t)-2}), (\ref{eq}), applying It\^o's formula to $\big[K(\cdot)x_1(\cdot)\big]x_1(\cdot)$ and putting $K(t)=0,t\in[0,T]$ into it, we get
\begin{equation*}\begin{aligned}
  0=&-\int_0^{r-\delta}x_1(t)x_1(t+\delta)\big[b_{x_\delta}(t+\delta)P(t+\delta)+\sigma_x(t+\delta)\sigma_{x_\delta}(t+\delta)P(t+\delta)\\
    &\quad+\sigma_{x_\delta}(t+\delta)Q(t+\delta)+b_{xx_\delta}(t+\delta)p(t+\delta)+\sigma_{xx_\delta}(t+\delta)q(t+\delta)+l_{xx_\delta}(t+\delta)\big]dt.
\end{aligned}\end{equation*}

Now choose $r=T$ and recall $x_1(t-\delta)=0$ for $t\in[0,\delta]$, we can derive
\begin{equation}\begin{aligned}\label{eqeq}
  0=&-\int_0^Tx_1(t)x_1(t-\delta)\big[b_{x_\delta}(t)P(t)+\sigma_x(t)\sigma_{x_\delta}(t)P(t)+\sigma_{x_\delta}(t)Q(t)\\
    &\qquad+b_{xx_\delta}(t)p(t)+\sigma_{xx_\delta}(t)q(t)+l_{xx_\delta}(t)\big]dt.
\end{aligned}\end{equation}
Substituting (\ref{eqeq}) into (\ref{variational inequality---}), we obtain
\begin{equation}\begin{aligned}\label{eqeqeq}
   &\mathbb{E}\int_0^T\Big\{\Delta l(t)+p(t)\Delta b(t)+q(t)\Delta\sigma(t)+\frac{1}{2}P(t)|\Delta\sigma(t)|^2\\
   &\qquad+x_1(t-\delta)\big[q(t)\Delta\sigma_{x_\delta}(t)+P(t)\sigma_{x_\delta}(t)\Delta\sigma(t)\big]\\
   &\qquad+x_1(t)\big[q(t)\Delta\sigma_x(t)+P(t)\Delta b(t)+P(t)\sigma_x(t)\Delta\sigma(t)+Q(t)\Delta\sigma(t)\big]\Big\}dt\geq o(\varepsilon).
 \end{aligned}\end{equation}
Next we try to prove that
\begin{equation}\begin{aligned}\label{eqeqeqeq}
   &\mathbb{E}\int_0^T\Big\{x_1(t-\delta)\big[q(t)\Delta\sigma_{x_\delta}(t)+P(t)\sigma_{x_\delta}(t)\Delta\sigma(t)\big]+x_1(t)\big[q(t)\Delta\sigma_x(t)+P(t)\Delta b(t)\\
   &\qquad+P(t)\sigma_x(t)\Delta\sigma(t)+Q(t)\Delta\sigma(t)\big]\Big\}dt=o(\varepsilon).
 \end{aligned}\end{equation}
In fact, by the boundedness of $\sigma_x$, we have
\begin{equation}\begin{aligned}\label{estimate-6}
      &\mathbb{E}\int_0^T|x_1(t)q(t)\Delta\sigma_x(t)|dt
       \leq \mathbb{E}\bigg[\sup\limits_{0\leq t\leq T}|x_1(t)|\Big(\int_{E_\varepsilon}|q(t)|^2dt\Big)^{\frac{1}{2}}
       \Big(\int_{E_\varepsilon}|\Delta\sigma_x(t)|^2dt\Big)^{\frac{1}{2}}\bigg]\\
  \leq&\ C\varepsilon^{\frac{1}{2}}\bigg\{\mathbb{E}\Big[\sup\limits_{0\leq t\leq T}|x_1(t)|^2\Big]\bigg\}^{\frac{1}{2}}
       \bigg\{\mathbb{E}\int_{E_\varepsilon}|q(t)|^2dt\bigg\}^{\frac{1}{2}}=o(\varepsilon).
\end{aligned}\end{equation}
Recall the assumption \textbf{(A1)} (ii) and the estimate of the solution to SDDE (\ref{SDDE}), we obtain
\begin{equation}\begin{aligned}\label{estimate-7}
      &\mathbb{E}\int_0^T\big|x_1(t)P(t)\Delta b(t)\big|dt
       \leq \mathbb{E}\bigg[\sup\limits_{0\leq t\leq T}|x_1(t)|\Big(\int_{E_\varepsilon}|P(t)|^2dt\Big)^{\frac{1}{2}}
       \Big(\int_{E_\varepsilon}|\Delta b(t)|^2dt\Big)^{\frac{1}{2}}\bigg]\\
  \leq&\ \bigg\{\mathbb{E}\Big[\sup\limits_{0\leq t\leq T}|x_1(t)|^2\Big]\bigg\}^{\frac{1}{2}}
       \bigg\{\mathbb{E}\bigg[\int_{E_\varepsilon}|P(t)|^2dt\int_{E_\varepsilon}C\Big(1+|x(t)|^2+|x(t-\delta)|^2\\
      &\qquad+|x^\varepsilon(t)|^2+|x^\varepsilon(t-\delta)|^2\Big)dt\bigg]\bigg\}^{\frac{1}{2}}=o(\varepsilon).
\end{aligned}\end{equation}
The other terms of (\ref{eqeqeqeq}) can be verified in a similar method. Hence we can simplify (\ref{eqeqeq}) to
\begin{equation}\begin{aligned}\label{variational inequality--}
   \mathbb{E}\int_0^T\Big[\Delta l(t)+p(t)\Delta b(t)+q(t)\Delta\sigma(t)+\frac{1}{2}P(t)|\Delta\sigma(t)|^2\Big]dt\geq o(\varepsilon).
\end{aligned}\end{equation}

\begin{Remark}\label{rem4.2}
We emphasize here again that, to obtain the estimate (\ref{variational inequality--}), the assumption $K(t)\equiv 0,t\in[0,T]$ plays an important role in this process, especially in (\ref{eqeqeq}) and (\ref{eqeqeqeq}). Thus the additional adjoint process $K(\cdot)$ satisfying the BRDE (\ref{BRDE}) is by no means trivial.
\end{Remark}

Now, we define the Hamilton function $H:[0,T]\times\mathbf{R}\times\mathbf{R}\times\mathbf{R}\times\mathbf{R}\times\mathbf{R}\times\mathbf{R}\times\mathbf{R}\rightarrow\mathbf{R}$ as
\begin{equation}\begin{aligned}\label{Hamilton}
   &H(\tau,x,x_\delta,v,v_\delta,p,q,P)\\
  =&\ l(\tau,x,x_\delta,v,v_\delta)+pb(\tau,x,x_\delta,v,v_\delta)+q\sigma(\tau,x,x_\delta,v,v_\delta)+\frac{1}{2}P|\sigma(\tau,x,x_\delta,v,v_\delta)|^2\\
   &-P\sigma(\tau,\Theta(\tau))\sigma(\tau,x,x_\delta,v,v_\delta),
\end{aligned}\end{equation}
where $\Theta(\tau)\equiv(x(\tau),x(\tau-\delta),u(\tau),u(\tau-\delta))$ and $(u(\cdot),x(\cdot))$ is the optimal pair. Then, we obtain the following main result in this paper.

\begin{mythm}\label{thm4.1}
Let assumption \textbf{(A1)} hold. Suppose $(u(\cdot),x(\cdot))$ is the optimal pair. Let $l_{x_\delta}(t)=l_{x_\delta x_\delta}(t)=0$ hold for $t\in(T,T+\delta]$. Suppose $(p(\cdot),q(\cdot))\in\mathcal{S}_\mathcal{F}^2([0,T];\mathbf{R})\times L_\mathcal{F}^2([0,T];\mathbf{R})$ and $(P(\cdot),Q(\cdot))\in \mathcal{S}_\mathcal{F}^2([0,T];\mathbf{R})\times L_\mathcal{F}^2([0,T];\mathbf{R})$ satisfy the ABSDEs (\ref{adjoint equations-1}) and (\ref{adjoint equations-2}), respectively. Besides, suppose $K(\cdot)$ satisfies the BRDE (\ref{BRDE}) with $K(t)=0$ for all $t\in[0,T]$. Then the following maximum condition holds:
\begin{equation}\begin{aligned}\label{main result}
      &H(\tau,x(\tau),x(\tau-\delta),v,u(\tau-\delta),p(\tau),q(\tau),P(\tau))\\
      &+\mathbb{E}^{\mathcal{F}_\tau}\big[H(\tau+\delta,x(\tau+\delta),x(\tau),u(\tau+\delta),v,p(\tau+\delta),q(\tau+\delta),P(\tau+\delta))\big]I_{[0,T-\delta)}(\tau)\\
 \geq &\ H(\tau,\Theta(\tau),p(\tau),q(\tau),P(\tau))\\
      &+\mathbb{E}^{\mathcal{F}_\tau}\big[H(\tau+\delta,\Theta(\tau+\delta),p(\tau+\delta),q(\tau+\delta),P(\tau+\delta))\big]I_{[0,T-\delta)}(\tau),\\
      &\qquad\qquad\qquad\qquad \forall v\in\textbf{U},\ a.e.\ \tau\in[0,T],\ \mathbb{P}\mbox{-}a.s.,
\end{aligned}\end{equation}
where $H$ is the Hamiltonian function defined by (\ref{Hamilton}).
\end{mythm}

\begin{proof}
Let $x^\varepsilon(\cdot)$ be the state trajectory corresponding to $u^\varepsilon(\cdot)$ defined in (\ref{perturbed control}). Under the assumptions of Theorem \ref{thm4.1}, we can get (\ref{variational inequality--}). Noting $u^\varepsilon(t)\neq u(t)$ for $t\in[\tau,\tau+\varepsilon]$ and $u^\varepsilon(t-\delta)\neq u(t-\delta)$ for $t\in[\tau+\delta,\tau+\delta+\varepsilon]$, dividing both sides of (\ref{variational inequality--}) simultaneously by $\varepsilon$, we obtain
\begin{equation*}\begin{aligned}
\mathbb{E}&\Big\{l(\tau,x(\tau),x(\tau-\delta),v(\tau),u(\tau-\delta))-l(\tau,\Theta(\tau))\\
          &+p(\tau)\big[b(\tau,x(\tau),x(\tau-\delta),v(\tau),u(\tau-\delta))-b(\tau,\Theta(\tau))\big]\\
          &+q(\tau)\big[\sigma(\tau,x(\tau),x(\tau-\delta),v(\tau),u(\tau-\delta))-\sigma(\tau,\Theta(\tau))\big]\\
          &+\frac{1}{2}P(\tau)\big[|\sigma(\tau,x(\tau),x(\tau-\delta),v(\tau),u(\tau-\delta))|^2-|\sigma(\tau,\Theta(\tau))|^2\big]\\
          &-P(\tau)\sigma(\tau,\Theta(\tau))\big[\sigma(\tau,x(\tau),x(\tau-\delta),v(\tau),u(\tau-\delta))-\sigma(\tau,\Theta(\tau))\big]\Big\}\\
+\mathbb{E}&\Big\{l(\tau+\delta,x(\tau+\delta),x(\tau),u(\tau+\delta),v(\tau))-l(\tau+\delta,\Theta(\tau+\delta))\\
           &+p(\tau+\delta)[b(\tau+\delta,x(\tau+\delta),x(\tau),u(\tau+\delta),v(\tau))-b(\tau+\delta,\Theta(\tau+\delta))]\\
           &+q(\tau+\delta)[\sigma(\tau+\delta,x(\tau+\delta),x(\tau),u(\tau+\delta),v(\tau))-\sigma(\tau+\delta,\Theta(\tau+\delta))]\\
           &+\frac{1}{2}P(\tau+\delta)[|\sigma(\tau+\delta,x(\tau+\delta),x(\tau),u(\tau+\delta),v(\tau))|^2-|\sigma(\tau+\delta,\Theta(\tau+\delta))|^2]\\
           &-P(\tau+\delta)\sigma(\tau+\delta,\Theta(\tau+\delta))\big[\sigma(\tau+\delta,x(\tau+\delta),x(\tau),u(\tau+\delta),v(\tau))\\
           &\quad-\sigma(\tau+\delta,\Theta(\tau+\delta))]\Big\}I_{[0,T-\delta)}(\tau)\geq 0,\quad a.e.\ \tau\in[0,T].
\end{aligned}\end{equation*}
Choose $v(\tau)=vI_A+u(\tau)I_{A^c}$, $A\in\mathcal{F}_\tau$, $v\in\mathbf{U}$, then we get
\begin{equation*}\begin{aligned}
                    \mathbb{E}&\Big\{l(\tau,x(\tau),x(\tau-\delta),v,u(\tau-\delta))-l(\tau,\Theta(\tau))\\
                              &+p(\tau)\big[b(\tau,x(\tau),x(\tau-\delta),v,u(\tau-\delta))-b(\tau,\Theta(\tau))\big]\\
                              &+q(\tau)\big[\sigma(\tau,x(\tau),x(\tau-\delta),v,u(\tau-\delta))-\sigma(\tau,\Theta(\tau))\big]\\
                              &+\frac{1}{2}P(\tau)\big[|\sigma(\tau,x(\tau),x(\tau-\delta),v,u(\tau-\delta))|^2-\sigma^2(\tau,\Theta(\tau))\big]\\
                              &-P(\tau)\sigma(\tau,\Theta(\tau))\big[\sigma(\tau,x(\tau),x(\tau-\delta),v,u(\tau-\delta))-\sigma(\tau,\Theta(\tau))\big]\Big\}\\
+\mathbb{E}^{\mathcal{F}_\tau}&\Big\{l(\tau+\delta,x(\tau+\delta),x(\tau),u(\tau+\delta),v)-l(\tau+\delta,\Theta(\tau+\delta))\\
                              &+p(\tau+\delta)\big[b(\tau+\delta,x(\tau+\delta),x(\tau),u(\tau+\delta),v)-b(\tau+\delta,\Theta(\tau+\delta))\big]\\
                              &+q(\tau+\delta)\big[\sigma(\tau+\delta,x(\tau+\delta),x(\tau),u(\tau+\delta),v)-|\sigma(\tau+\delta,\Theta(\tau+\delta))|^2\big]\\
                              &+\frac{1}{2}P(\tau+\delta)\big[|\sigma(\tau+\delta,x(\tau+\delta),x(\tau),u(\tau+\delta),v)|^2-|\sigma(\tau+\delta,\Theta(\tau+\delta))|^2\big]\\
                              &-P(\tau+\delta)\sigma(\tau+\delta,\Theta(\tau+\delta))\big[\sigma(\tau+\delta,x(\tau+\delta),x(\tau),u(\tau+\delta),v)\\
                              &\quad-\sigma(\tau+\delta,\Theta(\tau+\delta))\big]\Big\}I_{[0,T-\delta)}(\tau)\geq 0,\quad a.e.\ \tau\in[0,T],\ \mathbb{P}\mbox{-}a.s.
\end{aligned}\end{equation*}
Hence (\ref{main result}) holds. The proof is complete.
\end{proof}

\begin{Remark}\label{rem4.3}
(1) Comparing with the classical global maximum principle by Peng \cite{Peng90}, the maximum condition (\ref{main result}) has an indicator function, which is the characteristic of the stochastic optimal control problem with delay.

(2) Noting in the above theorem, it is rigorous that $l_{x_\delta}(t)=l_{x_\delta x_\delta}(t)=0$ holds for $t\in(T,T+\delta]$. However, in \cite{CW10}, they assume that $l$ doesn't contain the state delay term $x_\delta$ (See equation (4) on page 1075 of \cite{CW10}). Thus in this paper we fill a gap.

(3) One may find that the Hamilton function $H$ defined in (\ref{Hamilton}) does not contain the adjoint variables $Q(\cdot)$ and $K(\cdot)$ in an explicit form. In fact, it implicitly depends on $K(\cdot)$, since the value of $P(\cdot)$ is relative to that of $K(\cdot)$ which could be easily seen from the derivation of the main result.
\end{Remark}

\section{Multi-dimensional case}

In this section, we simply present the corresponding results for the multi-dimensional case, without proof whose details are left to the interested readers.  Let $B(\cdot)=(B^1(\cdot),\cdots,B^d(\cdot))^\top$ and $v(\cdot)=(v^1(\cdot),\cdots,v^k(\cdot))^\top$. Consider the optimal control problem with the state equation (\ref{controlled SDDE}) and the cost functional (\ref{cost functional}), where
\begin{equation*}\left\{\begin{aligned}
  &b(t,x,x_\delta,v,v_\delta)=\left(
  \begin{array}{c}
  b^1(t,x,x_\delta,v,v_\delta)\\
  \vdots\\
  b^n(t,x,x_\delta,v,v_\delta)\\
  \end{array}
\right),\\
  &\sigma(t,x,x_\delta,v,v_\delta)=(\sigma^1(t,x,x_\delta,v,v_\delta),\cdots,\sigma^d(t,x,x_\delta,v,v_\delta)),\\
  &\sigma^j(t,x,x_\delta,v,v_\delta)=\left(
  \begin{array}{c}
  \sigma^{1j}(t,x,x_\delta,v,v_\delta)\\
  \vdots\\
  \sigma^{nj}(t,x,x_\delta,v,v_\delta)\\
  \end{array}
\right),\ 1\leq j\leq d.
\end{aligned}\right.\end{equation*}

The corresponding variational equations (\ref{variational equation-1}), (\ref{variational equation-2}) are as follows:
\begin{eqnarray*}\left\{\begin{aligned}
    dx_1(t)=&\big[b_x(t)x_1(t)+b_{x_\delta}(t)x_1(t-\delta)+\Delta b(t)\big]dt\\
            &+\big[\sigma_x(t)x_1(t)+\sigma_{x_\delta}(t)x_1(t-\delta)+\Delta\sigma(t)\big]dB(t),\ t\geq0,\\
     x_1(t)=&\ 0,\ -\delta\leq t\leq 0,
\end{aligned}\right.\end{eqnarray*}
\begin{eqnarray*}\left\{\begin{aligned}
    dx_2(t)=&\big[b_x(t)x_2(t)+b_{x_{\delta}}(t)x_2(t-\delta)+\frac{1}{2}b_{xx}(t)x_1(t)x_1(t)\\
            &\quad+\frac{1}{2}b_{x_\delta x_\delta}(t)x_1(t-\delta)x_1(t-\delta)+b_{xx_\delta}(t)x_1(t)x_1(t-\delta)\big]dt\\
            &+\big[\sigma_x(t)x_2(t)+\sigma_{x_\delta}(t)x_2(t-\delta)+\frac{1}{2}\sigma_{xx}(t)x_1(t)x_1(t)\\
            &\quad+\frac{1}{2}\sigma_{x_\delta x_\delta}(t)x_1(t-\delta)x_1(t-\delta)+\sigma_{xx_\delta}(t)x_1(t)x_1(t-\delta))\\
            &\quad+\Delta\sigma_x(t)x_1(t)+\Delta\sigma_{x_\delta}(t)x_1(t-\delta)\big]dB(t),\ t\geq 0,\\
     x_2(t)=&\ 0,\ -\delta\leq t\leq 0,
\end{aligned}\right.\end{eqnarray*}
here we use the following notions:
\begin{equation*}\left\{\begin{aligned}
  &b_{xx}(t)x_1(t)x_1(t)\equiv\left(
  \begin{array}{c}
  \mbox{tr}\big[b^1_{xx}(t)x_1(t)x_1(t)^\top\big]\\
  \vdots\\
  \mbox{tr}\big[b^n_{xx}(t)x_1(t)x_1(t)^\top\big]\\
  \end{array}
\right),\\
  &\sigma^j_{xx}(t)x_1(t)x_1(t)\equiv\left(
  \begin{array}{c}
  \mbox{tr}\big[\sigma^{1j}_{xx}(t)x_1(t)x_1(t)^\top\big]\\
  \vdots\\
  \mbox{tr}\big[\sigma^{nj}_{xx}(t)x_1(t)x_1(t)^\top\big]\\
  \end{array}
\right),\ 1\leq j\leq d.
\end{aligned}\right.\end{equation*}
And, the corresponding adjoint equations (\ref{adjoint equations-1}), (\ref{adjoint equations-2}) and (\ref{BRDE}) are as follows:
\begin{equation}\left\{\begin{aligned}\label{adjoint equations-1-multi}
-dp(t)=&\bigg\{b_x(t)^\top p(t)+\sum_{j=1}^d(\sigma_x^j(t))^\top q(t)+l_x(t)+\mathbb{E}^{\mathcal{F}_t}\bigg[b_{x_\delta}(t+\delta)^\top p(t+\delta)\\
       &\ +\sum_{j=1}^d(\sigma_{x_\delta}^j(t+\delta))^\top q(t+\delta)+l_{x_\delta}(t+\delta)\bigg]\bigg\}dt-q(t)dB(t),\ t\in[0,T],\\
  p(T)=&\ h_x(x(T)),\ p(t)=0,\ t\in(T,T+\delta],\ q(t)=0,\ t\in[T,T+\delta].
\end{aligned}\right.\end{equation}
\begin{equation}\left\{\begin{aligned}\label{adjoint equations-2-multi}
-dP(t)=&\bigg\{b_x(t)^\top P(t)+P(t)b_x(t)+\sum_{j=1}^d(\sigma_x^j(t))^\top P(t)\sigma_x^j(t)+\sum_{j=1}^d(\sigma_x^j(t))^\top Q^j(t)\\
       &\ +\sum_{j=1}^dQ^j(t)\sigma_x^j(t)+\big\langle p(t),b_{xx}(t)\big\rangle+\mbox{tr}\big[q(t)^\top\sigma_{xx}(t)\big]+l_{xx}(t)\\
       &\ +\mathbb{E}^{\mathcal{F}_t}\bigg[\sum_{j=1}^d(\sigma_{x_\delta}^j(t+\delta))^\top P(t+\delta)\sigma_{x_\delta}^j(t+\delta)
        +\big\langle p(t+\delta),b_{x_\delta x_\delta}(t+\delta)\big\rangle\\
       &\ +\mbox{tr}\big[q(t+\delta)^\top\sigma_{x_\delta x_\delta}(t+\delta)\big]+l_{x_\delta x_\delta}(t+\delta)\bigg]\bigg\}dt-\sum_{j=1}^dQ^j(t)dB^j(t),\ t\in[0,T],\\
  P(T)=&\ h_{xx}(x(T)),\ P(t)=0,\ t\in(T,T+\delta],\ Q(t)=0,\ t\in[T,T+\delta],
\end{aligned}\right.\end{equation}
\begin{equation}\left\{\begin{aligned}\label{BRDE-multi}
-dK(t)=&\bigg\{\frac{1}{2}b_{x_\delta}(t)^\top\big[P(t)+P(t)^\top\big]+\frac{1}{2}\sum_{j=1}^d(\sigma_{x_\delta}^j(t))^\top\big[P(t)+P(t)^\top\big]\sigma_{x}^j(t)\\
       &\ +\frac{1}{2}\sum_{j=1}^d(\sigma_{x_\delta}^j(t))^\top\big[Q^j(t)+(Q^j(t))^\top\big]+\big\langle p(t),b_{xx_\delta}(t)\big\rangle\\
       &\ +\mbox{tr}\big[q(t)^\top\sigma_{xx_\delta}(t)\big]+l_{xx_\delta}(t)\bigg\}dt,\ t\in[0,T],\\
  K(t)=&\ 0,\ t\in[T,T+\delta].
\end{aligned}\right.\end{equation}
Similarly, we can obtain the variational inequality:
\begin{equation}\begin{aligned}\label{variational inequality---multi}
   \mathbb{E}\int_0^T\Big\{\Delta l(t)+\langle p(t),\Delta b(t)\rangle+\mbox{tr}[q(t)^\top\Delta\sigma(t)]
    +\frac{1}{2}\mbox{tr}\big[(\Delta\sigma(t))^\top P(t)(\Delta\sigma(t))\big]\Big\}dt\geq o(\varepsilon).
\end{aligned}\end{equation}
Now, we define the Hamilton function $H:[0,T]\times\mathbf{R}^n\times\mathbf{R}^n\times\mathbf{R}^k\times\mathbf{R}^k\times\mathbf{R}^n\times\mathbf{R}^{n\times d}\times\mathbf{R}^{n\times n}\rightarrow\mathbf{R}$ as
\begin{equation}\label{Hamilton-multi}\begin{aligned}
   &H(\tau,x,x_\delta,v,v_\delta,p,q,P)\\
  =&\ l(\tau,x,x_\delta,v,v_\delta)+\big\langle p,b(\tau,x,x_\delta,v,v_\delta)\big\rangle+\mbox{tr}\big[q^\top\sigma(\tau,x,x_\delta,v,v_\delta)\big]\\
   &+\frac{1}{2}\mbox{tr}\Big\{\big[\sigma(\tau,x,x_\delta,v,v_\delta)-\sigma(\tau,x(\tau),x(\tau-\delta),u(\tau),u(\tau-\delta))\big]^\top P\\
   &\qquad\quad\times\big[\sigma(\tau,x,x_\delta,v,v_\delta)-\sigma(\tau,x(\tau),x(\tau-\delta),u(\tau),u(\tau-\delta))\big]\Big\}.
\end{aligned}\end{equation}
The following theorem is a multidimensional version of Theorem \ref{thm4.1}.
\begin{mythm}\label{thm5.1}
Let assumption \textbf{(A1)} hold. Suppose $(u(\cdot),x(\cdot))$ is the optimal pair. Let $l_{x_\delta}(t)=l_{x_\delta x_\delta}(t)=0$ hold for $t\in(T,T+\delta]$. Suppose $(p(\cdot),q(\cdot))\in\mathcal{S}_\mathcal{F}^2([0,T];\mathbf{R}^n)\times L_\mathcal{F}^2([0,T];\mathbf{R}^{n\times d})$ and $(P(\cdot),Q(\cdot))\in \mathcal{S}_\mathcal{F}^2([0,T];\mathbf{R}^{n\times n})\times(L_\mathcal{F}^2([0,T];\mathbf{R}^{n\times n})^d$ satisfy (\ref{adjoint equations-1-multi}) and (\ref{adjoint equations-2-multi}), respectively. Besides, suppose $K(\cdot)$ satisfies the BRDE (\ref{BRDE-multi}) with $K(t)=0$ for all $t\in[0,T]$. Then the following maximum condition holds:
\begin{equation}\label{main result-multi}\begin{aligned}
      &H(\tau,x(\tau),x(\tau-\delta),v,u(\tau-\delta),p(\tau),q(\tau),P(\tau))\\
      &+\mathbb{E}^{\mathcal{F}_\tau}\big[H(\tau+\delta,x(\tau+\delta),x(\tau),u(\tau+\delta),v,p(\tau+\delta),q(\tau+\delta),P(\tau+\delta))\big]I_{[0,T-\delta)}(\tau)\\
 \geq &\ H(\tau,x(\tau),x(\tau-\delta),u(\tau),u(\tau-\delta),p(\tau),q(\tau),P(\tau))\\
      &+\mathbb{E}^{\mathcal{F}_\tau}\big[H(\tau+\delta,x(\tau+\delta),x(\tau),u(\tau+\delta),u(\tau),p(\tau+\delta),q(\tau+\delta),P(\tau+\delta))\big]I_{[0,T-\delta)}(\tau),\\
      &\qquad\qquad\qquad\qquad \forall v\in\textbf{U},\ a.e.\ \tau\in[0,T],\ \mathbb{P}\mbox{-}a.s.
\end{aligned}\end{equation}
\end{mythm}

To finish this section, let us consider two special cases.

Case (i). When the coefficients $b,\sigma,l$ do not contain the delay terms. In this case (\ref{main result-multi}) reduces to
\begin{equation}\label{Peng1990}\begin{aligned}
      &H(\tau,x(\tau),v,p(\tau),q(\tau),P(\tau))\\
 \geq &\ H(\tau,x(\tau),u(\tau),p(\tau),q(\tau),P(\tau)),\quad \forall v\in\textbf{U},\ a.e.\ \tau\in[0,T],\ \mathbb{P}\mbox{-}a.s.,
\end{aligned}\end{equation}
where the Hamilton function $H:[0,T]\times\mathbf{R}^n\times\mathbf{R}^k\times\mathbf{R}^n\times\mathbf{R}^{n\times d}\times\mathbf{R}^{n\times n}\rightarrow\mathbf{R}$ is defined as
\begin{equation*}\begin{aligned}
   &H(\tau,x,v,p,q,P)=l(\tau,x,v)+\big\langle p,b(\tau,x,v)\big\rangle+\mbox{tr}\big[q^\top\sigma(\tau,x,v)\big]\\
   &+\frac{1}{2}\mbox{tr}\Big\{\big[\sigma(\tau,x,v)-\sigma(\tau,x(\tau),u(\tau))\big]^\top P\big[\sigma(\tau,x,v)-\sigma(\tau,x(\tau),u(\tau))\big]\Big\}.
\end{aligned}\end{equation*}
This is a well known result in Peng \cite{Peng90} or Yong and Zhou \cite{YZ99}.

Case (ii). The control domain is convex, and the coefficients $b,\sigma,l$ are continuously differentiable with respect to $v$. In this case, from (\ref{main result-multi}) we could obtain
\begin{equation}\label{ChenWu2010}\begin{aligned}
      &\big\langle H_v(\tau)+\mathbb{E}^{\mathcal{F}_\tau}\big[H_{v_\delta}(\tau+\delta)\big]
      I_{[0,T-\delta)}(\tau),v-u(\tau)\big\rangle\geq0,\quad \forall v\in\textbf{U},\ a.e.\ \tau\in[0,T],\ \mathbb{P}\mbox{-}a.s.,
\end{aligned}\end{equation}
where $H_\phi(\tau)\equiv H_\phi(\tau,x(\tau),x(\tau-\delta),u(\tau),u(\tau-\delta),p(\tau),q(\tau),P(\tau))$ for $\phi=v,v_\delta$, $\forall\tau\in[0,T]$, and the Hamilton function $H:[0,T]\times\mathbf{R}^n\times\mathbf{R}^n\times\mathbf{R}^k\times\mathbf{R}^k\times\mathbf{R}^n\times\mathbf{R}^{n\times d}\rightarrow\mathbf{R}$ is defined as
\begin{equation*}\begin{aligned}
   &H(\tau,x,x_\delta,v,v_\delta,p,q)\\
  =&\ l(\tau,x,x_\delta,v,v_\delta)+\big\langle p,b(\tau,x,x_\delta,v,v_\delta)\big\rangle+\mbox{tr}\big[q^\top\sigma(\tau,x,x_\delta,v,v_\delta)\big].
\end{aligned}\end{equation*}
This coincides with the result in Chen and Wu \cite{CW10}.

\section{A solvable LQ example}

In this section, we present an linear-quadratic (LQ for short) example, to illustrate the global maximum principle obtained in the previous section. For simplicity, we consider $n=k=d=1$.

Consider the following linear controlled stochastic system with delay:
\begin{equation}\left\{\begin{aligned}\label{controlled SDDE-LQ}
      dX(t)&=\big[Mv(t)+\bar{M}v(t-\delta)\big]dt+\big[CX(t-\delta)+Dv(t)+\bar{D}v(t-\delta)\big]dB(t),\ t\in[0,T],\\
  X(\theta)&=\varphi(\theta),\ v(\theta)=\eta(\theta),\ \theta\in[-\delta,0],
\end{aligned}\right.\end{equation}
with the quadratic cost functional
\begin{equation}\label{cost functional-LQ}
  J(v(\cdot))=\mathbb{E}\bigg\{\frac{1}{2}\int_0^T\big[N|v(t)|^2+\bar{N}|v(t-\delta)|^2\big]dt+X(T)\bigg\},
\end{equation}
where $M,\bar{M},C,D,\bar{D},N,\bar{N}\in\mathbf{R}$ and $N,\bar{N}>0$.

The admissible control set is defined as $\mathcal{U}_{ad}:=\big\{v(\cdot):[0,T]\times\Omega\rightarrow\mathbf{U}|v(\cdot)$ is an $\mathcal{F}_t$-predictable, square-integrable process\big\}, where $\mathbf{U}=(-\infty,-1]\cup[1,\infty)$ is a nonconvex set. Our object is to find an optimal control $u(\cdot)\in\mathcal{U}_{ad}$ such that (\ref{controlled SDDE-LQ}) is satisfied and (\ref{cost functional-LQ}) is minimized.

The adjoint equations (\ref{adjoint equations-1}), (\ref{adjoint equations-2}) and (\ref{BRDE}) now read:
\begin{equation}\left\{\begin{aligned}\label{eq5.3}
  -dp(t)&=C\mathbb{E}^{\mathcal{F}_t}[q(t+\delta)]dt-q(t)dB(t),\ t\in[0,T],\\
    p(T)&=1,\ q(T)=0,\ p(t)=q(t)=0,\ t\in(T,T+\delta],
\end{aligned}\right.\end{equation}
\begin{equation}\left\{\begin{aligned}\label{eq5.4}
  -dP(t)&=C^2\mathbb{E}^{\mathcal{F}_t}[P(t+\delta)]dt-Q(t)dB(t),\ t\in[0,T],\\
    P(T)&=0,\ Q(T)=0,\ P(t)=Q(t)=0,\ t\in(T,T+\delta],
\end{aligned}\right.\end{equation}
\begin{equation}\left\{\begin{aligned}\label{eq5.5}
  -dK(t)&=CQ(t)dt,\ t\in[0,T],\\
    K(t)&=0,\ t\in[T,T+\delta].
\end{aligned}\right.\end{equation}
which admit $\mathcal{F}_t$-adapted unique solutions $(p(\cdot),q(\cdot)),(P(\cdot),Q(\cdot))$ and $K(\cdot)$, respectively. It is easy to check that for $t\in[0,T+\delta]$,
\begin{equation}\left\{\begin{aligned}\label{pqPQK}
  &p(t)\equiv 1,\ q(t)\equiv 0,\\
  &P(t)\equiv 0,\ Q(t)\equiv 0,\ K(t)\equiv 0.
\end{aligned}\right.\end{equation}
The Hamilton function $H$ of (\ref{Hamilton}) is now reduced to:
\begin{equation}\begin{aligned}\label{Hamilton-LQ}
  &H(\tau,x,x_\delta,v,v_\delta,p,q,P)\\
  &=Nv^2+\bar{N}v^2_\delta+p(Mv+\bar{M}v_\delta)+q(Cx_\delta+Dv+\bar{D}v_\delta)+\frac{1}{2}P(Cx_\delta+Dv+\bar{D}v_\delta)^2\\
  &\quad-P\big[Cx(\tau-\delta)+Du(\tau)+\bar{D}u(\tau-\delta)\big](Cx_\delta+Dv+\bar{D}v_\delta),
\end{aligned}\end{equation}
where $(u(\cdot),x(\cdot))$ is the optimal pair. According to Theorem \ref{thm4.1}, we obtain
\begin{equation}\begin{aligned}\label{main result-LQ}
  &\big[N+\bar{N}I_{[0,T-\delta)}(\tau)\big]v^2+\big[M+\bar{M}I_{[0,T-\delta)}(\tau)\big]v\\
  &\geq\big[N+\bar{N}I_{[0,T-\delta)}(\tau)\big]|u(\tau)|^2+\big[M+\bar{M}I_{[0,T-\delta)}(\tau)\big]u(\tau),\ \forall v\in\mathbf{U},\ a.e.\ \tau\in[0,T],\ \mathbb{P}\mbox{-}a.s.
\end{aligned}\end{equation}
From (\ref{main result-LQ}), it is easy to obtain that the optimal control is of the form as
\begin{equation}\label{optimal control-LQ}
  u(\tau)=-\frac{M+\bar{M}I_{[0,T-\delta)}(\tau)}{2\big[N+\bar{N}I_{[0,T-\delta)}(\tau)\big]},\ a.e.\ \tau\in[0,T].
\end{equation}
Especially, if we let $M=\bar{M}=2$, $N=\bar{N}=1$, then from (\ref{optimal control-LQ}) we know that $u(\tau)\equiv-1,\ a.e.\ \tau\in[0,T]$. Noting $\mathbf{U}=(-\infty,-1]\cup[1,\infty)$ is nonconvex, this optimal control can not be given by the results in the previous literatures.

\section{Concluding remarks}

In this paper, inspired by \cite{OS00} and \cite{CW10}, we have proved the global maximum principle for the stochastic optimal control problem with delay, where the control domain is nonconvex and the diffusion term contains both control and its delayed terms. We have introduced a new pair of adjoint variables $(P(\cdot),Q(\cdot))$ satisfying a second-order ABSDE, and another new adjoint variable $K(\cdot)$ satisfying a BRDE. Especially $K(\cdot)$ is introduced mainly to deal with the cross terms of $x_1(\cdot)$ and $x_1(\cdot-\delta)$ (see Remark \ref{rem4.1}). Comparing with the classical stochastic maximum principle without delay, the maximum condition contains an indicator function. In fact it is the characteristic of the stochastic optimal control problem with delay. Furthermore, it is interesting to find that the Hamilton function $H$ does not explicitly contain the new adjoint variable $K(\cdot)$, and it is relative to the adjoint variables $p(\cdot)$, $q(\cdot)$ and $P(\cdot)$. An LQ example is given to illustrate the main results obtained in this paper.

Possible extension to obtain the global maximum principles for controlled forward-backward stochastic systems or recursive utilities (Hu \cite{Hu17}) with delay (Chen and Wu \cite{CW09}, Chen and Huang \cite{CH15}, Huang, Li and Shi \cite{HLS12}, Shi, Xu and Zhang \cite{SXZ15}), is quite worth to study. Stochastic differential games with delay (Chen and Yu \cite{CY15}, Paman \cite{Pa15}, Xu et al. \cite{XSZ18}) and applications, are also hot research topics recently. We wish to deal with these problems in our future research.


\end{document}